\documentclass[11pt]{amsart} 

\usepackage{amsmath, amsthm, amscd, amsfonts, amssymb, enumerate, verbatim, newlfont, calc, graphicx, color}
\usepackage[bookmarksnumbered, colorlinks, plainpages]{hyperref}
\usepackage{tikz}
 \usepackage{doi}
 
\usepackage{amsmath, amsthm, amscd, amsfonts, amssymb, enumerate,verbatim, newlfont, calc, graphicx, color}
\usepackage[bookmarksnumbered, colorlinks, plainpages]{hyperref}
\usepackage{tikz} 
\usepackage{doi}
\usepackage{cancel}
\newtheorem{theorem}{Theorem}[section]

\newtheorem{lemma}[theorem]{Lemma}
\newtheorem{proposition}[theorem]{Proposition}

\newtheorem{corollary}[theorem]{Corollary}
\newtheorem{definition}[theorem]{Definition}

\newtheorem{remark}[theorem]{Remark}

\newtheorem{example}[theorem]{Example}

\usepackage{mathtools}

\usepackage{pstricks}
\usepackage{epsfig}
\usepackage{pst-grad}
\usepackage{pst-plot}
\usepackage{subfig}

\usepackage{listings}

\begin{document}

\author[M. Nasernejad   and   J. Toledo]{Mehrdad  Nasernejad$^{1,*}$ and   Jonathan Toledo$^{2}$}
\title[On the copersistence property  of monomial ideals]{On the copersistence property and nearly copersistence property of monomial ideals}

\thanks{$^*$Corresponding author}


\thanks{E-mail addresses:  m$\_$nasernejad@yahoo.com  and  jonathan.tt@itvalletla.edu.mx}  
\maketitle

\begin{center}
{\it
$^{1}$Univ. Artois, UR 2462, Laboratoire de Math\'{e}matique de  Lens (LML), \\  F-62300 Lens, France \\ and \\
Universit\'e de Caen, GREYC CNRS UMR-6072, Campus II,\\
 Bd Marechal Juin BP 5186, 14032 Caen cedex, France\\
$^2$Tecnol\'o{g}ico Nacional de M\'e{x}ico, Instituto Tecnol\'o{g}ico Del Valle de Etla, Abasolo S/N, Barrio Del Agua Buena, Santiago
Suchilquitongo, 68230, Oaxaca, M\'e{x}ico
}
\end{center}

\maketitle

\begin{abstract}
Let $I$ be an ideal in  a commutative Noetherian ring  $R$.  We say  that $I$  has the {copersistence property} if 
$\mathrm{Ass}_R(R/I^k) \supseteq  \mathrm{Ass}_R(R/I^{k+1})$ for all $k\geq 1$. The first aim of this paper is to explore this concept for 
monomial ideals such that we introduce some classes of monomial ideals satisfying this definition and also studying  this notion under monomial operations.  
In addition, we say that   $I$  has  {nearly copersistence property}  if there exist a positive integer $s$ and a monomial prime ideal  $\mathfrak{p}$ such that 
 $\mathrm{Ass}_R(R/I^{m}) \cup \{\mathfrak{p}\} \supseteq \mathrm{Ass}_R(R/I^{m+1})$ for all $1\leq m\leq s$, and  $\mathrm{Ass}_R(R/I^m) \supseteq \mathrm{Ass}_R(R/I^{m+1})$   for all $m \geq s+1$. The second goal of this paper is to investigate monomial ideals satisfying this concept. 
\end{abstract}


\section{Introduction and Overview}

Broadly speaking, monomial ideals play a fundamental role in studying the connection between commutative algebra and combinatorics. Indeed, the relationship between these two fields allows us to use techniques and methods from commutative algebra to solve combinatorial problems, and vice versa. As a result, commutative algebraists have begun exploring the properties of finite simple (hyper)graphs through monomial ideals.
Generally, there are two well-known correspondences between combinatorics and square-free monomial ideals, both of which arise from identifying square-free monomials with sets of vertices in either a simplicial complex or a hypergraph. In particular, for a finite simple graph, one can define both the edge ideal and the cover ideal. Let $G=(V(G), E(G))$ be a finite simple graph on the vertex set $[n]:=\{1,\ldots,n\}$.  Recall from \cite{Villarreal}  that the {\it edge ideal }  associated to $G$ is the monomial ideal
$$I(G)=(x_ix_j \ : \ \{i,j\}\in E(G)) \subset R=K[x_1,\ldots, x_n] ,$$ and 
the {\it cover ideal} associated to $G$ is the monomial ideal $$J(G)=\bigcap_{\{i,j\}\in E(G)}(x_i,x_j)\subset R=K[x_1,\ldots, x_n].$$
Over the past decade, valuable papers on combinatorial commutative algebra have been published, enabling discussions on the powers of square-free monomial ideals when these are viewed as edge ideals or cover ideals of hypergraphs. In \cite{FHV1, FHV2}, C. Francisco, H. T. H\`a, and A. Van Tuyl provided two methods for determining the chromatic number of a hypergraph via an ideal-membership problem, one involving secant ideals and the other involving powers of the cover ideal. Specifically, they described how the associated primes of the square of the cover ideal of a graph can detect its odd induced cycles. Furthermore, they demonstrated how the associated primes of the $s$-th power of the cover ideal of a hypergraph can be interpreted in terms of the coloring properties of its $s$-th expansion hypergraph. Notably, in the case of graphs, they provided two algebraic characterizations of perfect graphs that are independent of the Strong Perfect Graph Theorem. More information can be found in \cite{FHM}.

From the perspective of commutative algebra,  let  $R$ be a commutative Noetherian ring and $I$ be an ideal of $R$. 
Then a prime ideal $\mathfrak{p}\subset  R$ is an {\it associated prime} of $I$ if there exists an element $f$ in $R$ such that $\mathfrak{p}=(I:_R f)$, where $(I:_R f)=\{r\in R \mid  rf\in I\}$. In addition,  the  {\it set of associated primes} of $I$, denoted by  $\mathrm{Ass}_R(R/I)$, is the set of all prime ideals associated to  $I$. A well-known result of  Brodmann \cite{BR} states  that the sequence $\{\mathrm{Ass}_R(R/I^s)\}_{s \geq 1}$ of associated prime ideals is stationary  for large $s$, that is, there exists a positive integer $s_0$ such that $\mathrm{Ass}_R(R/I^s)=\mathrm{Ass}_R(R/I^{s_0})$ for all $s\geq s_0$. The  minimal such $s_0$ is called the {\it index of stability}   of  $I$ and $\mathrm{Ass}_R(R/I^{s_0})$ is called the {\it stable set }  of associated prime ideals of  $I$, which is denoted by $\mathrm{Ass}^{\infty }(I).$ 

Brodmann's result has been a source of inspiration for many definitions in commutative algebra. One of the most interesting  notions is  the persistence property. 
 We say  $\ell_0$ is the \textit{persistence index} of $I$ if $\ell_0$ is the smallest integer such that $\mathrm{Ass}_R(R/I^{\ell})\subseteq \mathrm{Ass}_R(R/I^{\ell+1})$ for all  $\ell\geq \ell_0$.  We say  an ideal $I$ of $R$  satisfies the {\it persistence property} if  $\ell_0=1$. 
 The concept  of the persistence property of ideals has been studied over the last decade, cf. \cite{HQ,KNT, MMV, N2}.

More recently, Heuberger et al. \cite{HRR} introduced the notion of the copersistence index. In fact,  a positive integer $k_0$ is called  the \textit{copersistence index} of $I$ if  $k_0$  is the smallest integer such that $\mathrm{Ass}_R(R/I^k) \supseteq  \mathrm{Ass}_R(R/I^{k+1})$ for all $k\geq k_0$.  
 Particularly,  we say that  $I$  has the \textit{copersistence property} if $\mathrm{Ass}_R(R/I^k) \supseteq  \mathrm{Ass}_R(R/I^{k+1})$ for all $k\geq 1$, that is, $k_0=1$.  The first target of this paper is to explore this concept for monomial ideals such that we introduce some classes of monomial ideals satisfying this definition. In particular, we  study this notion under monomial operations. These monomial operations enable us to construct new monomial ideals which have the copersistence property based on the monomial ideals which we know have the copersistence property.
 
 In the sequel,  we turn our attention to  the notion of nearly copersistence property. In  2020,  Andrei-Ciobanu \cite{Andrei-Ciobanu} introduced the definition of 
   nearly normally torsion-free monomial ideals.  A monomial ideal $I$ in a polynomial  ring $R=K[x_1, \ldots, x_n]$ over a field $K$ is called {\it nearly normally torsion-free}  if there exist a positive integer $k$ and a monomial prime ideal  $\mathfrak{p}$ such that $\mathrm{Ass}_R(R/I^m)=\mathrm{Min}(I)$ for all $1\leq m\leq k$, and  $\mathrm{Ass}_R(R/I^m) \subseteq \mathrm{Min}(I) \cup \{\mathfrak{p}\}$ for all $m \geq k+1$, more information can be found in
    \cite{NQKR}.     
        Inspired by this definition, we introduce the concept of nearly copersistence property.     
Let  $I$ be a monomial ideal in a polynomial ring  $R=K[x_1, \ldots, x_n]$, where $K$ is a  field. We say   $I$  has  \textit{nearly copersistence property} 
 if there exist a positive integer $s$ and a monomial prime ideal  $\mathfrak{p}$ such that 
 $\mathrm{Ass}_R(R/I^{m}) \cup \{\mathfrak{p}\} \supseteq \mathrm{Ass}_R(R/I^{m+1})$ for all $1\leq m\leq s$, and  $\mathrm{Ass}_R(R/I^m) \supseteq \mathrm{Ass}_R(R/I^{m+1})$   for all $m \geq s+1$. It is obvious that the copersistence property implies nearly copersistence property. 
 The second purpose of  this paper is to investigate monomial ideals that satisfying this concept. Particularly, we show that this type of monomial ideals appears in
 several well-known classes of monomial ideals. 

This paper is organized as follows. In Section \ref{Preliminaries}, for ease of reference, we collect the necessary results and facts   which will be used in the rest of this paper. 
Section \ref{monomial operations-COP}  is concerned with exploring the behavior of the  copersistence property under monomial operations.  In fact, we study the copersistence property under some monomial  operations, such as summation (Proposition \ref{Summation}), monomial multiple (Proposition  \ref{Multiple}), 
expansion (Theorem \ref{Th.Expansion}),   weighting (Theorem \ref{Th.Weighting}), monomial localization (Theorem \ref{Th.Localization}), 
  contraction (Corollary \ref{Cor. contraction}), deletion (Proposition \ref{Pro.Deletion}), and polarization (Examples \ref{Exam.Polarization.1} and \ref{Exam.Polarization.2}), and  by employing  some of them,  we introduce a lot of approaches  for constructing new monomial ideals which have the copersistence  property according to  the monomial ideals which have the copersistence property. In particular, in Theorem \ref{Th.Infinite. COP}, we demonstrate that there exist infinitely many monomial ideals possessing the copersistence property.

Section \ref{NCOP} deals with  nearly copersistence property of monomial ideals. We first show that if  $G=(V(G), E(G))$ is  a finite simple connected graph, and $J(G)$ denotes the cover ideal of $G$, then   $J(G)$ has nearly copersistence property  if and only if $G$ is either a bipartite graph or an almost bipartite graph, see Theorem \ref{NCOP.Almost}. After that, in Proposition \ref{NCOP-to-COP} and Lemma \ref{Lem.NCOP.1}, we seek some connections between the copersistence property and nearly copersistence property, and will present some applications of them in Examples \ref{Exam.NCOP.1} and \ref{Exam.NCOP.2}. Next, 
we prove that, under a certain condition,  a monomial ideal has nearly copersistence property if and only if its monomial multiple  has 
nearly  copersistence property, cf. Proposition \ref{Pro.NCOP.Multiple}.  In particular, we verify that a monomial ideal has nearly copersistence property if and only if its expansion (or weighted) has nearly copersistence property, refer to Lemma \ref{Lem.NCOP.Expansion+Weighting}. Ultimately, 
in Proposition \ref{Pro.Infinite.NCOP}, we  verify  that there exist infinitely many monomial ideals satisfying  nearly copersistence property.

Throughout this paper, the notation $\mathbb{N}$ stands for the positive  integers. Moreover,  $\mathcal{G}(I)$ denotes  the unique minimal set of monomial generators of  a monomial ideal $I\subset R=K[x_1, \ldots, x_n]$. The {\em support} of a monomial $u\in R=K[x_1, \ldots, x_n]$, denoted by $\mathrm{supp}(u)$, is the set of variables that divide $u$. Furthermore,  for a monomial ideal $I$, 
we set $\mathrm{supp}(I)=\bigcup_{u \in \mathcal{G}(I)}\mathrm{supp}(u)$.


\section{Preliminaries} \label{Preliminaries}

In what follows, we collect the necessary results, definitions,  and facts   which will be used in the rest of this paper. We  commence with the following lemma.

\begin{theorem} \label{ANKRQ} (\cite[Theorem 2.1]{ANKRQ}) Let $I$ be a normal monomial ideal in $R=K[x_1, \ldots, x_n]$ and $h\in R$ a monomial.
Assume $v\in R$  is a square-free monomial with  $\mathrm{gcd}(u, v) = 1$ for all $u\in  \mathcal{G}(I) \cup \{h\}$. 
Then  $L := I + vhR$ is normal if and only if $J := I + hR$ is normal.
\end{theorem}


\begin{lemma} (\cite[Lemma 2.11]{FHV2}) \label{FHV2}
  Let  $\mathcal{H}$ be a finite simple hypergraph on 
  $V = \{x_1, \ldots , x_n\}$ with cover ideal 
  $J(\mathcal{H}) \subseteq R=k[x_1, \ldots, x_n]$. 
  Then 
  $$P = (x_{i_1} , \ldots , x_{i_r}) \in \mathrm{Ass}(R/J(\mathcal{H})^d) \Leftrightarrow P = (x_{i_1} , \ldots , x_{i_r}) \in \mathrm{Ass}(k[P]/J(\mathcal{H}_P)^d),$$
where $k[P] =k[x_{i_1} , \ldots , x_{i_r}]$, and $\mathcal{H}_P$ is the induced hypergraph of $\mathcal{H}$ on the vertex set $P = \{x_{i_1} , \ldots , x_{i_r}\} \subseteq V$.  
\end{lemma}


\begin{corollary} \label{Cover. Bipartite. NTF} (\cite[Corollary 2.6]{GRV})
If $G$  is a bipartite graph and $J = I_c(G)$, then $\mathrm{gr}_J(R)$   is reduced.
\end{corollary}


  \begin{proposition}\label{NKA.Pro} (\cite[Proposition 3.6]{NKA})
 Suppose that $C_{2n+1}$ is  a cycle graph on the  vertex set $[2n+1]$, $R=K[x_1, \ldots, x_{2n+1}]$ is a  polynomial ring over a field $K$,
 and $\mathfrak{m}$ is the unique homogeneous maximal ideal  of $R$. Then  $$\mathrm{Ass}_R(R/(J(C_{2n+1}))^s)= \mathrm{Ass}_R(R/J(C_{2n+1}))\cup \{\mathfrak{m}\},$$  for all $s\geq 2$. In particular, 
 $$\mathrm{Ass}^\infty(J(C_{2n+1}))=\{(x_i, x_{i+1})~: ~ i=1, \ldots 2n\}\cup\{(x_{2n+1}, x_1)\}\cup \{\mathfrak{m}\}.$$
  \end{proposition}


\begin{proposition}\label{Pro.supp} (\cite[Proposition 4.2]{NKRT})
Let  $I$ be   a monomial ideal in $R=K[x_1, \ldots, x_n]$ over a field $K$ with $\mathcal{G}(I)=\{u_1, \ldots, u_m\}$ and $\mathrm{Ass}_R(R/I)=\{\mathfrak{p}_1, \ldots, \mathfrak{p}_s\}$. Then, the following statements hold. 
\begin{itemize}
\item[(i)] If $x_i|u_t$ for some $i$ with  $1\leq i \leq n$, and  
for some $t$ with  $1\leq t \leq m$, then there exists $j$ with  $1\leq j \leq s,$ 
such that $x_i\in \mathfrak{p}_j$. 
\item[(ii)] If $x_i\in \mathfrak{p}_j$ for some $i$ with  $1\leq i \leq n$, and  for some $j$ with   $1\leq j \leq s$, then  there exists $t$ with   $1\leq t \leq m$, such that  $x_i|u_t$. 
\end{itemize} 
Especially, $\bigcup_{j=1}^s \mathrm{supp}(\mathfrak{p}_j)=\bigcup_{t=1}^m \mathrm{supp}(u_t)$.
\end{proposition}


\begin{lemma} \label{NTF1} (\cite[Lemma 2.5]{NQH})
Let $I\subset R=K[x_1, \ldots, x_n]$  be a  normally torsion-free square-free monomial ideal,  $\mathfrak{q}$ be  a prime monomial ideal in $R$,  and  $h$ be a square-free monomial in $R$ with  $\mathrm{supp}(h) \cap  (\mathrm{supp}(\mathfrak{q}) \cup \mathrm{supp}(I))=\emptyset$  
 such that  $\bigcap_{\mathfrak{p}\in \mathrm{Ass}(I)}\mathfrak{p} \cap \bigcap_{x_r\in \mathrm{supp}(h)}(\mathfrak{q}, x_r)$   (respectively,  $\bigcap_{\mathfrak{p}\in \mathrm{Ass}(I)}\mathfrak{p} \cap \mathfrak{q}$)  is a minimal primary decomposition of $I \cap (\mathfrak{q}, h)$ (respectively,  $I \cap \mathfrak{q}$).  Let   $I$ and $I \cap \mathfrak{q}$ be normally torsion-free. Then 
$L:=I \cap (\mathfrak{q}, h)$  is normally torsion-free.  
\end{lemma}


\begin{proposition} \label{Roberts-Normal}  (\cite[Proposition 3.1]{RRV})
Let $I \subseteq  R = K[x_1,  \ldots, x_n]$  be a monomial ideal. If $I^m$ is integrally closed for $m = 1, \ldots, n-1$,  then $I$ is normal.
  \end{proposition}
  

\begin{theorem}\label{Th.NTF.Summation}(\cite[Theorem 2.5]{SN})
Let $I$ be  a  monomial ideal  in  $R=K[x_1, \ldots, x_n]$ such that $I=I_1R + I_2R$, where
 $\mathcal{G}(I_1) \subseteq R_1=K[x_1, \ldots, x_m]$ and $\mathcal{G}(I_2) \subseteq R_2=K[x_{m+1}, \ldots, x_n]$ for some  positive integer $m$. If $I_1$ and   $I_2$  are normally torsion-free, then  $I$  is so.
 \end{theorem}
 

\begin{theorem} (\cite[Theorem 3.21]{SN}) \label{Th.NTF.Deletion}. 
Let $I$ be a square-free monomial ideal in $R=K[x_1, \ldots, x_n]$, and $1\leq j \leq n$. If $I$ is normally torsion-free, then $I\setminus x_j$  is so.
\end{theorem}


\begin{lemma}\label{Ass-Local-Sharp} (\cite[Lemma  9.38]{sharp})
Let $M$  be a module over the commutative Noetherian ring
$R$, and let $S$ be a multiplicatively closed subset of $R$. Then
$$\mathrm{Ass}_{S^{-1}R}(S^{-1}M)=\{\mathfrak{p}S^{-1}R: \mathfrak{p}\in \mathrm{Ass}_R(M) ~~ \mathrm{and} ~~ \mathfrak{p} \cap S =\emptyset \}.$$
\end{lemma}


\begin{definition}
  \em{
   An   ideal $I$ in a commutative Noetherian ring $R$  has the {\it strong persistence property}  if $(I^{s+1}:_R I)=I^s$ for all  $s\geq 1$. 
   }
   \end{definition}
   
   \begin{definition}
   \em{
  An  ideal $I$ in a commutative Noetherian ring $R$ is called {\it normally torsion-free} 
 if $\mathrm{Ass}_R(R/I^s)\subseteq \mathrm{Ass}_R(R/I)$   for all $s\geq 1$.
 }
 \end{definition}



 \section{Copersistence property under monomial operations} \label{monomial operations-COP}
 
 In this section,  we explore  the behavior of the copersistence property under monomial operations. In fact, we study the copersistence
property under some monomial operations, such as expansion, weighting, monomial multiple, monomial localization, contraction, deletion, 
and polarization, and by using  some of them, we introduce several  methods for constructing new monomial ideals which have the copersistence property
based on  the monomial ideals which have the copersistence property. 
To accomplish this, we begin with the following  remark. 

 \begin{remark}  \label{Rem.Cop.1}
\em{
Let $I$ be an ideal in a  commutative Noetherian ring $R$ such that    $s_0$ is  the copersistence index  of $I$, this is, 
  $\mathrm{Ass}_R(R/I^s) \supseteq \mathrm{Ass}_R(R/I^{s+1})$ for all $s \geq s_0$. 
Then  one can easily see that    $\mathrm{Ass}_R(R/I^\lambda) \supseteq  \mathrm{Ass}_R(R/I^\theta)$ 
for all $\theta \geq \lambda \geq s_0$.
}
\end{remark}
 
\subsection{Copersistence property  under summation}  $\newline$

 \vspace{.1cm}
 Proposition \ref{Summation} states that, under a certain condition, the summation of two monomial ideals which have the copersistence property has the copersistence property as well. We need first  recall the following lemma. 

\begin{lemma} (\cite[Lemma 2.2]{Andrei-Ciobanu}) \label{Andrei-Ciobanu}
Let $I_1 \subset S_1=K[x_1,\ldots,  x_n]$,  $I_2 \subset S_2= K[y_1, \ldots, y_m]$ be two monomial ideals in disjoint sets of variables. 
 Let  $I = I_1S + I_2S \subset  S = K[x_1, \ldots, x_n, y_1, \ldots,  y_m]$.
Then $\mathfrak{p} \in \mathrm{Ass}_S(I^k)$ if and only if $\mathfrak{p}=\mathfrak{p}_1S + \mathfrak{p}_2S$,  where
 $\mathfrak{p}_1\in \mathrm{Ass}_{S_1}(I^{k_1})$ and  $\mathfrak{p}_2\in \mathrm{Ass}_{S_2}(I^{k_2})$
  for some positive integers $k_1, k_2$ with $k_1 + k_2 = k + 1$.
\end{lemma}


\begin{proposition} \label{Summation}
Let $I$ be a monomial ideal in $R=K[x_1, \ldots, x_n]$ such that 
$I=I_1R + I_2R$, where
$\mathcal{G}(I_1) \subset R_1=K[x_1, \ldots, x_m]$ and $\mathcal{G}(I_2) \subset R_2=K[x_{m+1}, \ldots, x_n]$ for some $m \geq 1$. 
 If $I_1$ and $I_2$ have the copersistence property,  then $I$ has the copersistence property. 
\end{proposition}
\begin{proof}
Suppose that  $I_1$ and $I_2$ have the copersistence property. Take  $\mathfrak{p}\in \mathrm{Ass}_R(R/I^{s+1})$, where $s\geq 1$.  
 It follows from Lemma \ref{Andrei-Ciobanu} that $\mathfrak{p}=\mathfrak{p}_1R + \mathfrak{p}_2R$,  where
 $\mathfrak{p}_1\in \mathrm{Ass}_{R_1}(R_1/I^{k_1})$ and  $\mathfrak{p}_2\in \mathrm{Ass}_{R_2}(R_2/I^{k_2})$
  for some positive integers $k_1$ and $ k_2$ with $k_1 + k_2 = s + 2$. Since $s\geq 1$, we must have  $k_1\geq 2$ or $k_2\geq 2$. 
  Without loss of generality, assume that $k_1 \geq 2$. 
 Because   $I_1$ has  the copersistence property, we get  $\mathfrak{p}_1\in \mathrm{Ass}_{R_1}(R_1/I^{k_1-1}_1)$. 
 Due to $k_1-1+k_2=s+1$, according to  Lemma \ref{Andrei-Ciobanu},  we can deduce that  $\mathfrak{p}\in \mathrm{Ass}_R(R/I^{s})$.
  Therefore, $I$ has the copersistence property, as required.  
\end{proof}


\subsection{Copersistence property under powers and monomial  multiple}$\newline$
 In this subsection, we focus on studying the behavior of the copersistence property  under powers and monomial multiple. 
  Before stating Proposition \ref{Power}, we  recollect the  following definition.

 \begin{definition} (\cite[Definition 11]{RT}) \label{Permutation}
 \em{
  Let $I\subset R=K[x_1, \ldots, x_n]$ be a square-free monomial ideal and   $\sigma$ a permutation on $\mathrm{supp}(I)$.
  Then we define the $\sigma(I)$ as the square-free monomial ideal in $R$ such that 
  $x_{i_1} \cdots x_{i_s}\in \mathcal{G}(I)$ if and only if   $x_{\sigma(i_1)} \cdots x_{\sigma(i_s)}\in \mathcal{G}(\sigma(I))$.  
  }
 \end{definition}

\begin{proposition} \label{Power}
Let $I$ be an ideal in a commutative Noetherian ring $R$. Then the following statements hold.
\begin{itemize}
\item[(i)] There exists a positive integer $s$ such that $I^s$ has the copersistence property. 
\item[(ii)]  If $I$  has the copersistence property, then, for all positive integers $s$, $I^s$
has the copersistence property. 
\item[(iii)] With the notation of Definition \ref{Permutation}, a square-free monomial ideal $I$ has the copersistence property if and only if $\sigma(I)$ 
has the copersistence property. 
 \item[(iv)]  If $I$ is an ideal such that $\mathrm{Ass}_R(R/I)=\mathrm{Min}(I)$, and also has the copersistence property, 
then $\mathrm{Ass}_R(R/I^s)= \mathrm{Ass}_R(R/I)=\mathrm{Min}(I)$ for all $s\geq 1$. In particular, $I$ is normally torsion-free. 
\item[(v)]  If $I$ is a square-free monomial ideal, and also has the copersistence property, then
 $I$ is normal, and hence has  both the strong persistence property and persistence property. 
\end{itemize}
\end{proposition}

\begin{proof}
(i)  Suppose that   $s_0$ is  the copersistence index  of $I$; this means that   $\mathrm{Ass}_R(R/I^s) \supseteq \mathrm{Ass}_R(R/I^{s+1})$ for all $s \geq s_0$. We want to verify that  $I^{s_0}$ has the copersistence property. Fix $k\geq 1$ and pick $\mathfrak{p}\in \mathrm{Ass}_R(R/(I^{s_0})^{k+1})$. 
  By Remark \ref{Rem.Cop.1}, we have    $\mathrm{Ass}_R(R/I^{s_0k}) \supseteq \mathrm{Ass}_R(R/I^{s_0k+s_0})$. This implies that $\mathfrak{p}\in \mathrm{Ass}_R(R/(I^{s_0})^{k})$, and hence $\mathrm{Ass}_R(R/(I^{s_0})^k) 
\supseteq  \mathrm{Ass}_R(R/(I^{s_0})^{k+1})$ for all $k\geq 1$.  Consequently,  $I^{s_0}$ has the copersistence property, as claimed. 

(ii)  Let $I$  have  the copersistence property.  Fix  $k,s \geq 1$ and choose $\mathfrak{p}\in \mathrm{Ass}_R(R/(I^s)^{k+1})$. 
 By virtue of  Remark \ref{Rem.Cop.1}, we obtain  $\mathrm{Ass}_R(R/I^{sk}) \supseteq  \mathrm{Ass}_R(R/I^{sk+s})$, and so 
  $\mathfrak{p}\in \mathrm{Ass}_R(R/(I^s)^{k})$. Accordingly, we get $I^s$ has the copersistence property.

(iii)  Let  $\mathrm{supp}(I)=\{x_1, \ldots, x_n\}$.  Consider the $K$-algebra homomorphism $\psi : R=K[x_1, \ldots, x_n] \rightarrow R$  given by $\psi(x_i)=x_{\sigma(i)}$ for all $i=1, \ldots, n$. It is not difficult to see that $\psi$ is an automorphism of $R$ with  $\psi(I)=\sigma(I)$. Thus, we obtain   $I$ and $\sigma(I)$  are isomorphic. In particular, this permits us to deduce that $I$ has the copersistence property if and only if $\sigma(I)$ has the copersistence property. 

(iv)   Since $I$ has the copersistence property, this gives that $\mathrm{Ass}_R(R/I^s) \supseteq \mathrm{Ass}_R(R/I^{s+1})$ for all $s \geq 1$.
 In particular, we have $\mathrm{Ass}_R(R/I^s) \subseteq \mathrm{Ass}_R(R/I)$ for all $s\geq 1$.  This implies that $I$ is normally torsion-free. 
  Furthermore,  it is well-known that $\mathrm{Min}(I) \subseteq \mathrm{Ass}_R(R/I^s)$ for all $s\geq 1$. Accordingly, we get 
 $\mathrm{Ass}_R(R/I^s)= \mathrm{Ass}_R(R/I)=\mathrm{Min}(I)$ for all $s\geq 1$.  

(v) As $I$ is a  square-free monomial ideal, we get $\mathrm{Ass}_R(R/I)=\mathrm{Min}(I)$.  Hence, one can deduce from part (iv) that $I$ 
is normally torsion-free. 
On the other hand, \cite[Theorem 1.4.6]{HH1} says that every  normally torsion-free square-free monomial ideal is normal. Moreover, 
\cite[Theorem 6.2]{RNA} states that every normal monomial ideal has the strong persistence property. Finally, in general, the strong persistence property
  implies the persistence property  (for example see  the proof of \cite[Proposition 2.9]{N1}).
\end{proof}


The next proposition says that  a monomial ideal has the copersistence property if and only if its monomial multiple  has the copersistence property.  
To show this claim,  one requires to use the following theorem. 

\begin{theorem} (\cite[Theorem 5.2] {KHN2})\label{5.2KHN}
Let $I$ be a monomial ideal of $R$ with $\mathcal{G}(I) =\{u_1,\ldots,u_m\}$. Also, assume that there exists a monomial $h=x_{j_1}^{b_1}\cdots x_{j_s}^{b_s}$ such that $h \mid  u_i$ for all $i=1,\ldots,m$. Set $J:=(u_1/h,\ldots,u_m/h)$. Then  we have      $\mathrm{Ass}_R(R/I)=\mathrm{Ass}_R(R/J)\cup\{ (x_{j_1}),\ldots,(x_{j_s})\}.$
\end{theorem}


\begin{proposition}\label{Multiple}
Let  $I \subset R=K[x_1, \ldots, x_n]$ be a monomial ideal and $h$ be a monomial in $R$. 
 Then  $I$ has the copersistence property if and only if $hI$ has the copersistence property.
 \end{proposition}
\begin{proof}
 Let $h=x_{j_1}^{b_1}\cdots x_{j_s}^{b_s}$ with $j_1, \ldots, j_s \in \{1, \ldots, n\}$. Fix $t\geq 1$. 
 According to Theorem \ref{5.2KHN}, we obtain the following 
 \begin{equation} \label{Multiple.1}
 \mathrm{Ass}_R(R/(hI)^{t+1})=\mathrm{Ass}_R(R/I^{t+1})\cup\{ (x_{j_1}),\ldots,(x_{j_s})\},
 \end{equation}
  and  
 \begin{equation} \label{Multiple.2}
 \mathrm{Ass}_R(R/(hI)^t)=\mathrm{Ass}_R(R/I^t)\cup\{ (x_{j_1}),\ldots,(x_{j_s})\}.
 \end{equation}
  To show the forward implication, let $I$ have the copersistence property. Hence, we have $\mathrm{Ass}_R(R/I^t) \supseteq  \mathrm{Ass}_R(R/I^{t+1})$. 
 It follows immediately from (\ref{Multiple.1}) and (\ref{Multiple.2}) that  $\mathrm{Ass}_R(R/(hI)^t) \supseteq  \mathrm{Ass}_R(R/(hI)^{t+1})$; 
 equivalently,   $hI$ has the copersistence property. Conversely, assume that  $hI$ has  the copersistence property and pick 
 $\mathfrak{p}\in \mathrm{Ass}_R(R/I^{t+1})$. One can rapidly deduce from  (\ref{Multiple.1}) 
 that $\mathfrak{p}\in \mathrm{Ass}_R(R/(hI)^{t+1})$. Since  $\mathrm{Ass}_R(R/(hI)^t) \supseteq  \mathrm{Ass}_R(R/(hI)^{t+1})$, we get  
  $\mathfrak{p}\in \mathrm{Ass}_R(R/(hI)^t)$. Hence, (\ref{Multiple.2}) gives  $\mathfrak{p}\in  \mathrm{Ass}_R(R/I^t)\cup\{ (x_{j_1}),\ldots,(x_{j_s})\}.$ 
If $\mathfrak{p}\in \mathrm{Ass}_R(R/I^t)$, then the proof is over. We thus assume $\mathfrak{p}=(x_{j_c})$ for some  $1\leq c \leq s$. 
This implies  that $\mathfrak{p} \in \mathrm{Min}(I^{t+1})$. Due to $\mathrm{Min}(I^t)= \mathrm{Min}(I^{t+1})$ and 
 $ \mathrm{Min}(I^{t}) \subseteq \mathrm{Ass}_R(R/I^t)$, we can derive  that $\mathfrak{p}\in \mathrm{Ass}_R(R/I^t)$. Consequently, 
 $I$ has the copersistence property, as desired.
\end{proof}


\subsection{Copersistence property under expansion}$\newline$

We first  recall the definition of the expansion operation  on monomial ideals, which has been presented  in \cite{BH}. \par 
Let $K$ be a field and $R = K[x_1, \ldots , x_n]$ be the polynomial ring over a field $K$ in the variables $x_1, \ldots , x_n$. Fix an ordered $n$-tuple $(i_1, \ldots , i_n)$ of positive integers, and consider the polynomial ring $R^{(i_1,\ldots ,i_n)}$ over $K$ in the variables $$x_{11}, \ldots , x_{1i_1} , x_{21}, \ldots , x_{2i_2} , \ldots , x_{n1}, \ldots , x_{ni_n}.$$
Let $\mathfrak{p}_j$ be the monomial prime ideal $(x_{j1}, x_{j2}, \ldots , x_{ji_j}) \subseteq R^{(i_1,\ldots,i_n)}$ for all $j=1,\ldots,n$. Attached to each monomial ideal $I \subset R$ a set of monomial generators $\{\mathbf {x}^{\mathbf{ a}_1} , \ldots , {\mathbf x}^{{\mathbf a}_m}\}$, where ${\mathbf x}^{\mathbf a_i}={x_1}^{a_i(1)}\cdots {x_n}^{a_i(n)}$ and $a_i(j)$ denotes the $j$th component of the vector ${\mathbf a}_i=(a_i(1),\ldots,a_i(n))$ for all $i=1,\ldots,m$. We define the {\it expansion of I with respect to the n-tuple $(i_1, \ldots, i_n)$}, denoted by $I^{(i_1,\ldots,i_n)}$, to be the monomial ideal
$$I^{(i_1,\ldots,i_n)} = \sum_{i=1}^m \mathfrak{p}_1^{a_i(1)}\cdots \mathfrak{p}_n^{a_i(n)}\subseteq R^{(i_1,\ldots,i_n)}.$$
We simply write $R^*$ and $I^*$,
respectively, rather than $R^{(i_1,\ldots,i_n)}$ and $I^{(i_1,\ldots,i_n)}$.\par
As an  example, assume that  $R = K[x_1, x_2, x_3]$ and the ordered $3$-tuple $(3,1,2)$. Then
we get $\mathfrak{p}_1 = (x_{11},x_{12}, x_{13})$, $\mathfrak{p}_2 = (x_{21})$, and $\mathfrak{p}_3 = (x_{31}, x_{32})$. This implies that 
for the monomial ideal $I = (x_1x_2^2,  x_2x_3, x_3^3)$, the ideal $I^* \subseteq K[x_{11}, x_{12}, x_{13}, x_{21}, x_{31}, x_{32}]$ is $\mathfrak{p}_1\mathfrak{p}^2_2+\mathfrak{p}_2\mathfrak{p}_3+\mathfrak{p}^3_3$, namely
$$I^* =(x_{11}x_{21}^2,x_{12}x_{21}^2,x_{13}x_{21}^2,x_{21}x_{31},x_{21}x_{32},x_{31}^3,x_{31}^2x_{32},x_{31}x_{32}^2,x_{32}^3).$$

\bigskip
To formulate Theorem \ref{Th.Expansion}, we have to utilize the  following auxiliary results. 
 
\begin{lemma} (\cite [Lemma 1.1]{BH}) \label{Lem.Bayati.Expansion}
Let $I$ and $J$ be monomial ideals in a polynomial ring $S$. Then
\begin{itemize}
\item[(i)] $f \in I^*$ if and only if $\pi(f)\in I$, for all $f\in S^*$;
\item[(ii)] $(I + J)^* = I^* + J^*$;
\item[(iii)] $(IJ)^* = I^*J^*$;
\item[(iv)] $(I \cap J)^* = I^*\cap J^*$;
\item[(v)] $(I : J)^* = (I^* : J^*)$;
\item[(vi)] $\sqrt{I^*} = (\sqrt{I})^*$;
\item[(vii)] If the monomial ideal $Q$ is $\mathfrak{p}$-primary, then $Q^*$ is
$\mathfrak{p}^*$-primary.
\end{itemize}
\end{lemma}


\begin{proposition} (\cite [Proposition 1.2]{BH}) \label{Pro.Bayati.Expansion}
Let $I$ be a monomial ideal, and consider an (irredundant) primary
decomposition $I = Q_1\cap\cdots \cap Q_m$ of $I$. Then 
$I^*={Q^*_1} \cap \cdots \cap {Q^*_m}$ is an (irredundant) primary decomposition of $I^*$.
In particular, $\mathrm{Ass}(S^*/I^*) = \{\mathfrak{p}^* : \mathfrak{p} \in \mathrm{Ass}(S/I)\}.$
\end{proposition}


 The next theorem asserts that a monomial ideal has the copersistence property if and only if its expansion  has the copersistence property.

\begin{theorem}\label{Th.Expansion} 
Let  $I \subset R=K[x_1, \ldots, x_n]$ be a monomial ideal. Then $I$ has the copersistence property if and only if $I^*$ has the copersistence property, 
where $I^*$ denotes the  expansion of $I$.  
\end{theorem}
\begin{proof}
 Let $I \subset R$ have  the copersistence property. Fix  $s \geq 1$ and take  $\mathfrak{q}\in \mathrm{Ass}_{R^*}(R^*/(I^*)^{s+1})$. 
From  Lemma \ref{Lem.Bayati.Expansion}(iii), we get  $(I^*)^{s+1}=(I^{s+1})^*$, and so  $\mathfrak{q}\in \mathrm{Ass}_{R^*}(R^*/(I^{s+1})^*)$. 
Due to  Proposition \ref{Pro.Bayati.Expansion}, we deduce that   $\mathfrak{p}\in \mathrm{Ass}_R(R/I^{s+1})$, where $\mathfrak{q}=\mathfrak{p}^*$. 
 Since  $I$ has the copersistence property, this implies that  $\mathfrak{p}\in \mathrm{Ass}_R(R/I^s)$. Using  Proposition \ref{Pro.Bayati.Expansion}, we have 
  $\mathfrak{q}\in \mathrm{Ass}_{R^*}(R^*/(I^s)^*)$. By  Lemma \ref{Lem.Bayati.Expansion}(iii), we have  $(I^*)^{s}=(I^{s})^*$, 
 and hence  $\mathfrak{q}\in \mathrm{Ass}_{R^*}(R^*/(I^*)^s)$.   Accordingly,  $I^*\subset R^*$ has the copersistence property.   
 
The converse can be proved  by a similar argument.  
\end{proof}


\subsection{Copersistence property under weighting} $\newline$

The purpose  of  this subsection is to study the copersistence property under weighting operation. To reach this aim, we first review the needed definition.

\begin{definition}
\em{
A \textit{weight} over a polynomial ring $R=K[x_1, \ldots, x_n]$ is a function $W: \{x_1, \ldots, x_n\} \rightarrow \mathbb{N}$ such that  $w_i=W(x_i)$ is called the {\it weight} of the variable $x_i$. For a monomial ideal $I\subset R $ and a weight $W$, we define the \textit{weighted ideal}, denoted 
by $I_W$, to be the ideal generated by $\{h(u) : u\in \mathcal{G}(I) \}$, where $h$ is the unique homomorphism $h: R \rightarrow R$ given by 
$h(x_i)= x_i^{w_i}$. 
}
\end{definition}


To see  an example, let  $R=K[x_1, x_2, x_3, x_4, x_5]$ and the monomial ideal  $I=(x_1^2x_4, x_2^4x_4^2x_5, x_1x_3^2x_5^3)$ in $R$. 
Also,  let $W : \{x_1, x_2, x_3, x_4, x_5\} \rightarrow \mathbb{N}$ be a weight over $R$ with $W(x_1)=3$, $W(x_2)=2$, $W(x_3)=2$, $W(x_4)=4$, and $W(x_5)=1$. One can easily check that  the weighted ideal $I_W$ is given by 
$$I_W=(x_1^6x_4^4, x_2^8x_4^8x_5, x_1^3x_3^4x_5^3).$$

\bigskip
To establish Theorem \ref{Th.Weighting}, one has to refer to the below lemmata. 

\begin{lemma}\label{LEM. Weighted} (\cite[Lemma 3.5]{SN})
Let $I$ and $J$ be two monomial ideals of a polynomial ring $R=K[x_1, \ldots, x_n]$, and $W$ a weight over $R$. Then the following statements hold. 
\begin {itemize}
\item[(i)] $(I+J)_W= I_W + J_W$;
\item[(ii)] $(IJ)_W= I_W J_W$;
\item[(iii)] $(I\cap J)_W= I_W \cap J_W$;
\item[(iv)] $(I :_RJ)_W= (I_W :_R J_W)$.
\end{itemize}
\end{lemma}


\begin{lemma}\label{ASS-Weighted} (\cite[Lemma 3.9]{SN})
Let $I$  be a monomial ideal in a polynomial ring $R=K[x_1, \ldots, x_n]$, and $W$ a weight over $R$. Then 
$\mathrm{Ass}_R(R/I_W)=\mathrm{Ass}_R(R/I).$
\end{lemma}
 

The subsequent theorem states that a monomial ideal has the copersistence property if and only if its weighted  has the copersistence property.   

\begin{theorem}\label{Th.Weighting}
Let $I \subset R=K[x_1, \ldots, x_n]$  be a  monomial ideal,  and $W$ a weight over $R$. Then $I$ has the copersistence property if and only if $I_W$ has the copersistence property, where $I_W$ denotes the weighted ideal.  
\end{theorem}
\begin{proof}
Assume $I$ has the copersistence property. Pick $\mathfrak{p}\in \mathrm{Ass}_R(R/I_W^{s+1})$, where $s\geq 1$. 
 Based on   Lemma \ref{LEM. Weighted}(ii), we have  $I_W^{s+1}=(I^{s+1})_W$, and so $\mathfrak{p}\in \mathrm{Ass}_R(R/(I^{s+1})_W)$. 
 It follows from  Lemma \ref{ASS-Weighted} that  $\mathfrak{p}\in \mathrm{Ass}_R(R/I^{s+1})$. On account of  $I$ has the copersistence property, 
 this yields that  $\mathfrak{p}\in \mathrm{Ass}_R(R/I^s)$.  According to  Lemma \ref{ASS-Weighted}, we get  $\mathfrak{p}\in \mathrm{Ass}_R(R/(I^s)_W)$. By virtue of  Lemma \ref{LEM. Weighted}(ii), we deduce that $I_W^{s}=(I^{s})_W$, and so $\mathfrak{p}\in \mathrm{Ass}_R(R/I_W^{s})$.
 This gives rise to   $I_W$ has the copersistence property, as claimed.   

A similar discussion can be used to  establish the converse.
\end{proof} 


In what follows, we will give an application of Theorems \ref{Th.Expansion} and \ref{Th.Weighting}. To do this, we first express the following example which 
is vital to frame  Theorem \ref{Th.Infinite. COP}.   We state it here for ease of reference.

\begin{example} \label{HNTT} (\cite[Example 3.1]{HNTT})  
 Let $A = k[x, y, z]$. For any integer $d\geq  2$, let  $I = (x^{d+2}, x^{d+1}y, xy^{d+1}, y^{d+2}, x^dy^2z)$. Then
\[
\mathrm{depth}(A/I^n) =
    \begin{dcases}
    0  & \text{if }  n\leq d-1,\\
1 & \text{if }   n\geq d. \\
                \end{dcases}
\]
\end{example}

The below theorem, as an application of Theorems \ref{Th.Expansion}  and \ref{Th.Weighting},  guarantees there exist infinitely many  monomial ideals  in $R=K[x_1, \ldots, x_n]$ such that  have  the copersistence property.

\begin{theorem} \label{Th.Infinite. COP}
Let $R=K[x_1, \ldots, x_n]$ be  a polynomial ring in $n$  variables, where $n\geq 3$,  with coefficients in a  field $K$. Then 
there exist infinitely many  monomial ideals  in $R$ such that  have  the copersistence property.
\end{theorem}

\begin{proof}
 Let  $L = (x_1^{d+2}, x_1^{d+1}x_2, x_1 x_2^{d+1}, x_2^{d+2}, x_1^dx_2^2x_3)\subset S=K[x_1, x_2, x_3]$, where $d\geq 2$.  
 Then one can easily check that 
$$L=(x_1^{d+2}, x_2) \cap (x_1^{d+1}, x_2^2)  \cap  (x_1^{d}, x_2^{d+1}) \cap  (x_1, x_2^{d+2})  \cap  (x_1^{d+1}, x_2^{d+1}, x_3).$$  
We thus get  $\mathrm{Ass}_S(S/L)=\{(x_1, x_2), (x_1, x_2, x_3)\}$. Since $\mathrm{Min}(L)=\{(x_1, x_2)\}$ 
and $\mathrm{Min}(L) = \mathrm{Min}(L^k)\subseteq \mathrm{Ass}_S(S/L^k)$ for all $k\geq  1$, we can deduce  that 
 $(x_1, x_2) \in \mathrm{Ass}_S(S/L^k)$ for all $k\geq 1$. 
 On the other hand,  it is well-known that,  for all $k\geq 1$,    $(x_1, x_2, x_3)\in \mathrm{Ass}_S(S/L^k)$  if and only if $\mathrm{depth}(S/L^k)=0$ 
    (see \cite[Exercise  2.2.11]{V1}). This allows us to conclude  from Example \ref{HNTT} that, for all $k\geq 1$, we have 
\[
\mathrm{Ass}_S(S/L^k) =
    \begin{dcases}
    \{(x_1, x_2), (x_1, x_2, x_3)\} & \text{if } k\leq d-1,\\
     \{(x_1, x_2)\} & \text{if } k \geq d. \\
                \end{dcases}
\]
 Hence,  $L$ has the copersistence property. Now, let $\mathfrak{p}_1:=(x_{i_1}, \ldots, x_{i_a})$, $\mathfrak{p}_2:=(x_{i_{a+1}}, \ldots, x_{i_b})$, and $\mathfrak{p}_3:=(x_{i_{b+1}}, \ldots, x_{i_c})$ with  $\mathrm{supp}(\mathfrak{p}_i) \cap  \mathrm{supp}(\mathfrak{p}_j)=\emptyset$ for all $1\leq i<j \leq 3$ and $\bigcup_{i=1}^{3}\mathrm{supp}(\mathfrak{p}_i)=\{x_1, \ldots, x_n\}$. 
 One can promptly deduce from Theorem \ref{Th.Expansion} that $L^*$ in $R=K[x_1, \ldots, x_n]$ has the copersistence property as well. Now, assume that  $I=(L^*)_W$ with $W(x_i)=\alpha_i$ such that $\alpha_i \geq 1$ for all $i=1, \ldots, n$. In  light of Theorem \ref{Th.Weighting}, we obtain $I$ 
has the copersistence property. In fact, there are infinitely many  such monomial ideals. This completes the proof. 
 \end{proof}


\subsection{Copersistence property under localization and contraction} $\newline$

In this subsection, we investigate the copersistence property under localization and contraction. To see this, we start with the following proposition. 

\begin{proposition}
Let $I$ be an  ideal  in a commutative Noetherian ring $R$. Then  $I$ has  the copersistence property if and only if $I_{\mathfrak{p}}$ has  the 
copersistence property for all $\mathfrak{p}\in \mathrm{Ass}_R(R/I)$, where $I_{\mathfrak{p}}$ denotes the localization of $I$  at  $\mathfrak{p}$. 
\end{proposition}

\begin{proof}
First, let   $I$ have the copersistence property. Select  $\mathfrak{q} \in \mathrm{Ass}(I_\mathfrak{p}^{k+1})$, where  $k\geq 1$ and 
$\mathfrak{p}\in \mathrm{Ass}_R(R/I)$. By virtue of Lemma \ref{Ass-Local-Sharp}, we get $\mathfrak{q}=\mathfrak{p}'R_{\mathfrak{p}}$, where 
$\mathfrak{p}'\in \mathrm{Ass}(I^{k+1})$. Since   $I$ has  the copersistence property, this implies that $\mathfrak{p}'\in \mathrm{Ass}(I^{k})$. 
Once again, Lemma \ref{Ass-Local-Sharp} yields that $\mathfrak{q}=\mathfrak{p}'R_{\mathfrak{p}}\in \mathrm{Ass}(I_\mathfrak{p}^{k})$. 
Accordingly, $I_{\mathfrak{p}}$ has  the copersistence property.

The converse can be concluded by a similar argument. This finishes the proof. 
\end{proof}


Now, we turn our attention to the monomial localization. Let  $I\subset R=K[x_1, \ldots, x_n]$ be  a monomial ideal, where $K$ is a field. We  denote by $V^*(I)$ the set of monomial prime ideals containing $I$. Let $\mathfrak{p}=(x_{i_1}, \ldots, x_{i_r})$ be a monomial prime ideal. Then the  {\it monomial localization} of $I$ with respect to $\mathfrak{p}$, denoted by $I(\mathfrak{p})$, is the ideal in the polynomial ring $R(\mathfrak{p})=K[x_{i_1}, \ldots, x_{i_r}]$  which is obtained from $I$ by using the $K$-algebra homomorphism $R\rightarrow R(\mathfrak{p})$  with $x_j\mapsto 1$  for all $x_j\notin \{x_{i_1}, \ldots, x_{i_r}\}$. 

\bigskip
The following lemma and theorem are necessary for us to prove Theorem \ref{Th.Localization}. We insert  them here for ease of reference.

  \begin{lemma}\label{LEM.Localization} (\cite[Lemma 3.13]{SN})
Let $I$ and $J$ be two monomial ideals   in $R=K[x_1, \ldots, x_n]$, and $\mathfrak{p}$ be a monomial prime ideal of $R$. Then the following 
statements hold.
\begin {itemize}
\item[(i)]  $(I+J)(\mathfrak{p})= I(\mathfrak{p}) + J(\mathfrak{p})$;
\item[(ii)] $(IJ)(\mathfrak{p})= I(\mathfrak{p})  J(\mathfrak{p})$;
\item[(iii)] $(I\cap J)(\mathfrak{p})= I(\mathfrak{p}) \cap  J(\mathfrak{p})$;
\item[(iv)] $(I :_RJ)(\mathfrak{p})= (I(\mathfrak{p}) :_{R(\mathfrak{p})}  J(\mathfrak{p}))$;
\item[(v)]  If $Q$ is a $\mathfrak{q}$-primary monomial ideal in $R$ with $ \mathfrak{q}\subseteq \mathfrak{p}$, then $Q(\mathfrak{p})$ is a $\mathfrak{q}$-primary monomial ideal in $R(\mathfrak{p})$.
\end{itemize}
 \end{lemma}
 

\begin{theorem}\label{Localization}  (\cite[Theorem 3.14]{SN})
Let $I$ be a  monomial ideal in   a polynomial ring $R=K[x_1,\ldots, x_n]$, and $\mathfrak{p}\in V^*(I)$. Then 
$$\mathrm{Ass}_{R(\mathfrak{p})}(R(\mathfrak{p})/I(\mathfrak{p}))=\{\mathfrak{q}~:~ \mathfrak{q}\in \mathrm{Ass}_{R}(R/I)~\mathrm{and}~ \mathfrak{q} \subseteq \mathfrak{p}\}.$$
\end{theorem}


The next theorem declares that  if a monomial ideal has the copersistence property,  then any monomial localization of it has the copersistence property as well.

\begin{theorem}\label{Th.Localization}
Let $I \subset R=K[x_1, \ldots, x_n]$  be a  monomial ideal. If  $I$ has the copersistence property, then   $I(\mathfrak{p})$  has the copersistence property.
\end{theorem}
\begin{proof}
Suppose that $I$ has the copersistence property. Let $s\geq 1$ and take 
$\mathfrak{q} \in  \mathrm{Ass}_{R(\mathfrak{p})}(R(\mathfrak{p})/I(\mathfrak{p})^{s+1})$. It follows from Lemma \ref{LEM.Localization}(ii) that 
 $I(\mathfrak{p})^{s+1}=I^{s+1}(\mathfrak{p})$. Hence, we get   $\mathfrak{q} \in  \mathrm{Ass}_{R(\mathfrak{p})}(R(\mathfrak{p})/I^{s+1}(\mathfrak{p}))$.   By virtue of   Theorem \ref{Localization}, we have $\mathfrak{q}\in \mathrm{Ass}_R(R/I^{s+1})$ and $\mathfrak{q} \subseteq \mathfrak{p}$. 
 Since $I$ has the copersistence property, this implies that $\mathfrak{q}\in \mathrm{Ass}_R(R/I^s)$. Due to  $\mathfrak{q}\in \mathrm{Ass}_R(R/I^s)$ and $\mathfrak{q} \subseteq \mathfrak{p}$, Theorem  \ref{Localization} allows us to conclude that  
$\mathfrak{q} \in  \mathrm{Ass}_{R(\mathfrak{p})}(R(\mathfrak{p})/I^s(\mathfrak{p}))$.  Once again,  Lemma \ref{LEM.Localization}(ii) yields  that 
 $I(\mathfrak{p})^{s}=I^{s}(\mathfrak{p})$. Consequently, we can derive  that   $I(\mathfrak{p})$ has the copersistence property, as required.
\end{proof}


In the following corollary, we provide an advantageous  application of Theorem  \ref{Th.Localization}. In fact, Theorem \ref{Th.Localization} enables us  to refute the
 copersistence property for some monomial ideals, and  assists  us in constructing  new counterexamples derived from  some well-known counterexamples.
 
 \begin{corollary}
Let $I$ be a monomial ideal in $R_1=K[x_1, \ldots, x_s]$ with $\mathcal{G}(I)=\{u_1, \ldots, u_m\}$, and $J_1, \ldots, J_m$ be  some monomial ideals    in $R_2=K[x_{s+1}, \ldots, x_n]$.  If $I$ does  not have the  copersistence property, then $L:=u_1J_1R+\cdots + u_mJ_mR$ does not have the copersistence
  property, where $R=K[x_1, \ldots, x_n]$.
\end{corollary}
 
 \begin{proof}
 Suppose, on the contrary, that $L$ has the copersistence property.   Set  $\mathfrak{p}:= (\bigcup_{i=1}^m\mathrm{supp}(u_i))R$. It is clear  that $\mathfrak{p}$ is a  monomial  prime  ideal of  $R$,  and also $L\subseteq \mathfrak{p}$. It follows now from   Theorem \ref{Th.Localization} that  $L(\mathfrak{p})$ 
has the copersistence property.  Due to  $L(\mathfrak{p})=I$, we can deduce that  $I$ has the copersistence property, a contradiction. Consequently, we get
   $L$ does not have the copersistence property, and the proof is complete. 
\end{proof}


 To formulate the next result, we require to recollect  the definition of the contraction operation. Assume that $I$ is a monomial ideal in $R=K[x_1, \ldots, x_n]$ with $\mathcal{G}(I)=\{u_1, \ldots, u_m\}$. For some $1\leq j \leq n$,  the  \textit{contraction} of $x_j$  from  $I$, denoted by $I/x_j$, is obtained by setting $x_j=1$ in  $u_i$ for each $i=1, \ldots, m$.   Because  the  contraction $I/x_j$ is exactly the monomial localization of $I$ with respect to $\mathfrak{p}=\mathfrak{m}\setminus\{x_j\}$, where $\mathfrak{m}=(x_1, \ldots, x_n)$ is the graded maximal ideal of $R=K[x_1,\ldots, x_n]$, then
 according to  Theorem \ref{Th.Localization}, we can promptly deduce  the following result.  
 
 \begin{corollary}\label{Cor. contraction}
 Let $I$ be  a    monomial ideal in  $R=K[x_1, \ldots, x_n]$, and     $1\leq j \leq n$.   If  $I$ has the copersistence property, then   $I/x_j$  has the copersistence property.
 \end{corollary}


\subsection{Copersistence property under deletion} $\newline$

We first recall the definition of the deletion operation. Assume that   $I\subset R=K[x_1, \ldots, x_n]$ is a  monomial ideal with 
$\mathcal{G}(I)=\{u_1, \ldots, u_m\}$.  Then the  \textit{deletion}  of $x_j$  from  $I$ with $1\leq j \leq n$, denoted by $I\setminus x_j$, is formed  by putting  $x_j=0$ in  $u_i$ for each $i=1, \ldots, m$.  

The following proposition states that if a square-free monomial ideal has the copersistence property, then it still retains the copersistence property under the deletion operation.

\begin{proposition}  \label{Pro.Deletion}
Let $I\subset  K[x_1, \ldots, x_n]$  be a square-free monomial ideal and $1\leq i \leq n$. If  $I$ has the copersistence property, then   $I\setminus x_i$ has  the  copersistence  property. 
\end{proposition}

\begin{proof}
Assume that $I$ has the copersistence property. Set $L:=I\setminus x_i$ and $S=R\setminus x_i$. Since $I$ is square-free, Proposition \ref{Power}(iv) implies that $I$ is normally torsion-free. In light of Theorem \ref{Th.NTF.Deletion}, we can deduce that  $L$ in $S$ is normally torsion-free,  and hence    $\mathrm{Ass}_S(S/L^k)=\mathrm{Ass}_S(S/L)=\mathrm{Min}(L)$ for all $k\geq 1$. This gives that $L$ has the copersistence property, as claimed.
\end{proof}

It should be observed  that we cannot remove the word  ``square-free" in Proposition \ref{Pro.Deletion}. We provide such a counterexample in the 
next  example.

\begin{example}
\em{
Let $I=(x_1x_4, x_2^2, x_1x_2, x_3x_4^2, x_3^2x_4, x_2x_3x_4, x_1x_3^2)$ be a monomial ideal in $R=K[x_1, x_2, x_3, x_4]$. 
We first show that $I$ has the copersistence property.  Using  \textit{Macaulay2} \cite{GS}, we detect that  $I$, $I^2$, and $I^3$ are integrally closed. 
 Since $R=K[x_1, x_2, x_3, x_4]$ is a polynomial ring in four  variables,  we can deduce from  Proposition  \ref{Roberts-Normal}  that $I$ is a normal monomial ideal. 
 Now, by the proof of Proposition  \ref{Power}(v), we derive that $I$ has the  persistence property. Furthermore, Using  \textit{Macaulay2} \cite{GS}, we get 
$$\mathrm{Ass}_R(R/I)=\{(x_1,x_2, x_3), (x_1,x_2,x_4), (x_2,x_3,x_4), (x_1,x_2,x_3,x_4)\}.$$ 
Since $I$ has the persistence property  and  $(x_1,x_2,x_3,x_4)\in \mathrm{Ass}_R(R/I)$, this implies that $(x_1,x_2,x_3,x_4)\in \mathrm{Ass}_R(R/I^k)$ for all $k\geq 1$. In addition, due to $$\mathrm{Min}(I^k)=\mathrm{Min}(I)=\{(x_1,x_2, x_3), (x_1,x_2,x_4), (x_2,x_3,x_4) \text{ for all } k\geq  1,$$ 
 this gives   $\{(x_1,x_2, x_3), (x_1,x_2,x_4), (x_2,x_3,x_4)\} \subseteq \mathrm{Ass}_R(R/I^k)$ for all $k\geq 1$. Consequently, we obtain $\mathrm{Ass}_R(R/I^k)=\mathrm{Ass}_R(R/I)=\mathrm{Min}(I)$ for all $k\geq 1$. 
This means that $I$ has the copersistence property. Now, one can easily check that   $L:=I\setminus x_3=(x_1x_4, x_2^2, x_1x_2)$. 
 Using  \textit{Macaulay2} \cite{GS}, we get  $(x_1, x_2, x_4) \in \mathrm{Ass}(L^2)\setminus \mathrm{Ass}(L).$
This yields that $I\setminus x_3$ does not have the copersistence property. This finishes our argument.
}
\end{example}


\subsection{Copersistence property under polarization} $\newline$

First, we start by reviewing the definition of the polarization operation.

\begin{definition}
\em{ 
Let $I$ be a monomial ideal in a polynomial ring $R=K[x_1, \ldots, x_n]$.  The process of  {\it  polarization}   replaces a power $x_i^t$
by a product of $t$ new variables $x_{(i,1)}\cdots x_{(i,t)}$, where  each  $x_{(i,j)}$ is called a  {\it shadow} of $x_i$. In particular, we  use $\widetilde{I^t}$ to denote the polarization of $I^t$, and  $S_t$ for the new polynomial ring in this polarization, and also we  use $\widetilde{w} $ to denote the polarization in $S_t$ of a monomial $w$ in $R=K[x_1, \ldots, x_n]$. Moreover, the {\it depolarization} of a monomial  ideal in $S_t$ is obtained by putting  $x_{(i,j)}=x_i$ for all $i,j$.
}
\end{definition}

In what follows, in Example \ref{Exam.Polarization.1}, we present a monomial ideal $I$ which has  the  copersistence  property, but  $\widetilde{I}$ 
does not have  the copersistence property. Next, in Example \ref{Exam.Polarization.2}, we give a monomial ideal  $L$ which does not have the copersistence property, 
but  $\widetilde{L}$ has the copersistence property. 


\begin{example} \label{Exam.Polarization.1}
\em{
Let $R=K[x,y,z,t]$ be the polynomial ring over a field $K$ and $I=(x^4, y^4z, x^3y, xy^3, x^2y^2zt)$.   One can easily check  that 
$$I=(x,y)^{4} \cap (x,z) \cap (x^3, y^3,z) \cap (x^3,y^3,t).$$
Hence, we get $\mathrm{Ass}_R(R/I)=  \{(x,y), (x,z), (x,y,z), (x,y,t)\}$.  Now, we claim that 
$\mathrm{Ass}_R(R/I^s)=  \{(x,y), (x,z), (x,y,z)\}$ for all $s\geq 2$. First note that since $(x,y), (x,z) \in \mathrm{Min}(I)$, this implies that 
$(x,y), (x,z) \in\mathrm{Ass}_R(R/I^s)$  for all $s\geq 2$. 
We here verify that $t\notin \mathrm{supp}(I^s)$ for all $s\geq 2$. Take an arbitrary monomial 
$u\in \mathcal{G}(I^s)$, where $s\geq 2$.  Then there exist nonnegative integers $\ell_i$ ($i=1, \ldots, 5$) with $\sum_{i=1}^{5}\ell_i=s$ such that 
$$u=(x^4)^{\ell_1} (y^4z)^{\ell_2}  (x^3y)^{\ell_3}  (xy^3)^{\ell_4} (x^2y^2zt)^{\ell_5}.$$  We want to demonstrate that  there 
exists some monomial $v\in I^{s}$ with  $t \nmid v$ such that $v \mid u$.
If $\ell_5=0$, then there is  nothing to show. Let $\ell_5\geq 1$. Now,  we may consider the following cases:
  
\bigskip
\textbf{Case 1.} $\ell_5=2k$ for some $k\geq 1$.  Then it is not hard to check  that 
$$(x^4)^{\ell_1+k} (y^4z)^{\ell_2+k}  (x^3y)^{\ell_3}  (xy^3)^{\ell_4} \mid u.$$

 \textbf{Case 2.} $\ell_5=2k+1$ for some $k\geq 0$. It is easy to investigate the following statements
  \begin{itemize}
\item[(1)]   $x^4(x^2y^2z)=(x^3y)^2z\in I^2$.
\item[(2)]   $(y^4z)(x^2y^2z)=(xy^3)^2z^2\in I^2$.
\item[(3)]    $(x^3y)(x^2y^2z)=x^4(xy^3)z\in I^2$.
\item[(4)]    $(xy^3)(x^2y^2z)=(x^3y)(y^4z)\in I^2$. 
\end{itemize}
 
Put $u_1:=(x^4)^{\ell_1} (y^4z)^{\ell_2}  (x^3y)^{\ell_3}  (xy^3)^{\ell_4} (x^2y^2zt)^{2k}$. Then $u=u_1(x^2y^2zt)$. 
According to Case 1, there exists some monomial $v_1\in I^{s-1}$ with  $t \nmid v_1$ such that $v_1 \mid u_1$. It follows now from the above 
statements that $v_1(x^2y^2z)\in I^s$ with  $t\nmid v_1(x^2y^2z)$ such that $v_1(x^2y^2z) \mid u$.   

 Consequently, we get $t\notin \mathrm{supp}(I^s)$ for all $s\geq 2$. From Proposition \ref{Pro.supp}, we can conclude that, for all $s\geq 2$, 
 we have $(x,y,z,t) \notin \mathrm{Ass}_R(R/I^s)$, $(x,y,t) \notin \mathrm{Ass}_R(R/I^s)$, $(x,z,t) \notin \mathrm{Ass}_R(R/I^s)$, 
 and $(y,z,t) \notin \mathrm{Ass}_R(R/I^s)$.
 
 We here establish $(x,y,z) \in\mathrm{Ass}_R(R/I^s)$  for all $s\geq 2$. To do this, let $s\geq 2$ and  set 
$h:=x^{s+1}y^{3s-1}$. We show that $(I^s:h)=(x,y,z)$. According to the following statements

  \begin{itemize}
\item[(1)]   $xh=x^{s+2}y^{3s-1}=(x^3y)(xy^3)^{s-1}y\in I^s$;
\item[(2)]   $yh=x^{s+1}y^{3s}=x(xy^3)^s \in I^s$; 
\item[(3)]   $zh=x^{s+1}y^{3s-1}z=(x^3y)(y^4z)(xy^3)^{s-2}\in I^s$,
\end{itemize}
we can derive that $(x,y,z)\subseteq (I^s:h)$.  On the other hand, since $t\notin \mathrm{supp}(I^s)$ for all $s\geq 2$ and $(x,y), (x,z) \in\mathrm{Ass}_R(R/I^s)$  for all $s\geq 2$, we can rapidly deduce that $\mathrm{supp}(I^s)=\{x,y,z\}$ for all $s\geq 2$. In particular, this implies that $(I^s:h)\subseteq (x,y,z)$, and so $(I^s:h)=(x,y,z)$ for all $s\geq 2$.  

  Hence,  $\mathrm{Ass}_R(R/I^s)=  \{(x,y), (x,z), (x,y,z)\}$ for all $s\geq 2$. Therefore, we obtain    $\mathrm{Ass}_R(R/I^s) \supseteq  \mathrm{Ass}_R(R/I^{s+1})$ for all $s\geq 1$. This gives rise to   $I$ has the copersistence property.

In what follows, our aim is to show that  $\widetilde{I}$ does not have the copersistence property. To accomplish this,
 let $S=K[x_1, x_2, x_3, x_4, y_1,y_2, y_3, y_4, z_1, t_1]$  such that $\widetilde{x^4}=x_1x_2x_3x_4$, 
 $\widetilde{y^4}=y_1y_2y_3y_4$, $\widetilde{z}=z_1$, and $\widetilde{t}=t_1$.
    It is routine to investigate that  the polarization of $I$ in $S$ is as follows
$$\widetilde{I}=(x_1x_2x_3x_4, y_1y_2y_3y_4z_1, x_1x_2x_3y_1,  x_1y_1y_2y_3, x_1x_2y_1y_2z_1t_1).$$
 Using \textit{Macaulay2} \cite{GS}, we obtain
  $(x_1, x_3, y_3, z_1)\in \mathrm{Ass}_S(S/(\widetilde{I})^2) \setminus \mathrm{Ass}_S(S/\widetilde{I}).$
 This implies that $\widetilde{I}$ does not have the copersistence property, as claimed. 
 }
\end{example}

\begin{example} \label{Exam.Polarization.2}
\em{
Assume that $R=K[x_1, x_2, x_3, x_4]$ is the polynomial ring over a field $K$ and   $L=(x_1^3x_2, x_1x_2^3, x_2^4, x_1^4x_3, x_1^4x_4)$. 
 Using \textit{Macaulay2} \cite{GS}, we get 
  $(x_1, x_2, x_3, x_4)\in \mathrm{Ass}_R(R/L^2) \setminus \mathrm{Ass}_R(R/L).$  This gives that $L$  does not have the copersistence property. 
  
  From now on, we want to show that  $\widetilde{L}$ has  the copersistence property. 
 For this purpose, let $S=K[x_1, x_2, x_3, x_4, x_5, x_6, x_7, x_8, x_9, x_{10}]$ such that $\widetilde{x^4_1}=x_1x_2x_3x_4$, 
 $\widetilde{x^4_2}=x_5x_6x_7x_8$, $\widetilde{x_3}=x_9$, and $\widetilde{x_4}=x_{10}$. 
  Then  the polarization of $L$ in $S$ is as follows
 $$\widetilde{L}=(x_1x_2x_3x_5, x_1x_5x_6x_7, x_5x_6x_7x_8,  x_1x_2x_3x_4x_9, x_1x_2x_3x_4x_{10}).$$
 To simplify the notation, set $J:=\widetilde{L}$.  Using \textit{Macaulay2} \cite{GS}, we obtain 
 \begin{align*}
 \mathrm{Ass}_S(S/J)=\{&(x_1,x_5), (x_1,x_6), (x_1,x_7), (x_1, x_8), (x_2,x_5), (x_2,x_6), (x_2,x_7), \\  
 & (x_3,x_5),  (x_3,x_6), (x_3,x_7), (x_4,x_5),  (x_5, x_9, x_{10})\}.
 \end{align*}
We here  claim that $J$ is a  normally torsion-free square-free monomial ideal. To do this, put $\mathfrak{q}:=(x_5, x_9)$, $h=(x_{10})$, and 
 \begin{align*}
 I:&=(x_1,x_5) \cap (x_1,x_6) \cap  (x_1,x_7) \cap  (x_1, x_8) \cap  (x_2,x_5) \cap  (x_2,x_6)    \\  
 & \cap  (x_2,x_7) \cap  (x_3,x_5) \cap   (x_3,x_6) \cap  (x_3,x_7) \cap  (x_4,x_5).
 \end{align*}
 It is obvious that  $J=I \cap (\mathfrak{q}, h)$. Now, consider the following simple graphs
 \begin{align*}
 E(G)=\{&\{x_1,x_5\}, \{x_1,x_6\}, \{x_1,x_7\}, \{x_1, x_8\}, \{x_2,x_5\}, \{x_2,x_6\},    \\  
 &   \{x_2,x_7\}, \{x_3,x_5\}, \{x_3,x_6\}, \{x_3,x_7\}, \{x_4,x_5\}\},
 \end{align*}
  and 
\begin{align*}
 E(H)=\{&\{x_1,x_5\}, \{x_1,x_6\}, \{x_1,x_7\}, \{x_1, x_8\}, \{x_2,x_5\}, \{x_2,x_6\},    \\  
 &   \{x_2,x_7\}, \{x_3,x_5\}, \{x_3,x_6\}, \{x_3,x_7\}, \{x_4,x_5\}, \{x_5, x_9\}\}.
 \end{align*}
 One can easily  see that $I$ (resp. $I\cap \mathfrak{q}$) is the cover  ideal of  $G$ (resp.  $H$). 
 Since both $G$ and $H$ are bipartite, by  Corollary  \ref{Cover. Bipartite. NTF},  we obtain  both $I$ and $I\cap \mathfrak{q}$ are normally torsion-free. 
 In addition, we have  $\mathrm{supp}(h) \cap  (\mathrm{supp}(\mathfrak{q}) \cup \mathrm{supp}(I))=\emptyset$. It follows immediately from 
Lemma  \ref{NTF1} that $J$ is normally torsion-free as well, and so $\mathrm{Ass}_S(S/J^k)= \mathrm{Ass}_S(S/J)= \mathrm{Min}(J)$ for all $k\geq 1$. 
Therefore, $\widetilde{L}$  has the copersistence property, as required. 
}
\end{example}


\section{Nearly copersistence property of monomial ideals} \label{NCOP}

In this section, we aim to investigate the notion of nearly copersistence property of monomial ideals. Recall that a monomial ideal $I$ in a ploynomial ring 
  $R=K[x_1, \ldots, x_n]$ over a  field $K$ is said to be   \textit{nearly copersistence property} 
 if there exist a positive integer $s$ and a monomial prime ideal  $\mathfrak{p}$ such that 
 $\mathrm{Ass}_R(R/I^{m}) \cup \{\mathfrak{p}\} \supseteq \mathrm{Ass}_R(R/I^{m+1})$ for all $1\leq m\leq s$, and  $\mathrm{Ass}_R(R/I^m) \supseteq \mathrm{Ass}_R(R/I^{m+1})$   for all $m \geq s+1$. 

We begin this section with the following theorem which gives a large class of square-free monomial ideals which have nearly copersistence property. To do this, one 
needs to recall that a finite simple undirected connected graph $G$ is said to be  an {\it almost bipartite} graph when  $G$ has only one induced odd cycle subgraph.

\begin{theorem} \label{NCOP.Almost}
Let  $G=(V(G), E(G))$ be  a finite simple connected graph, and $J(G)$ denotes the cover ideal of $G$.  Then   $J(G)$ has 
nearly copersistence property  if and only if $G$ is either a bipartite graph or an almost bipartite graph. 
\end{theorem}

\begin{proof}
We  first assume that  $J(G)$  has  nearly copersistence property. On the contrary, assume  that $G$ is neither bipartite nor almost bipartite. This implies  $G$ has at least two induced odd cycle subgraphs, say $C_1$ and $C_2$.  Now,  Proposition  \ref{NKA.Pro} allows us to  deduce  
 $\mathfrak{p}_1=(x_i~:~ x_i \in  V(C_1))\in \mathrm{Ass}(J(C_1)^s)$ and  $\mathfrak{p}_2=(x_i~:~ x_i\in  V(C_2))\in \mathrm{Ass}(J(C_2)^s)$ for all $s\geq 2$. By virtue of  $G_{\mathfrak{p}_1}=C_1$ and $G_{\mathfrak{p}_2}=C_2$, it follows immediately from  Lemma \ref{FHV2}
  that $\mathfrak{p}_1, \mathfrak{p}_2 \in \mathrm{Ass}_R(R/J(G)^s)$ for all $s\geq 2$, where   $R=K[x_i : x_i \in V(G)]$. 
  Since $J(G)$ has nearly copersistence property, there exists some   monomial prime ideal $\mathfrak{p}$ such that 
   $\mathrm{Ass}_R(R/J(G)) \cup \{\mathfrak{p}\}\supseteq \mathrm{Ass}_R(R/J(G)^2)$.   We know that if $\mathfrak{q}\in \mathrm{Ass}_R(R/J(G))$, then $\mathrm{supp}(\mathfrak{q})=2$. On account of   $\mathrm{supp}(\mathfrak{p}_1)\geq 3$   and $\mathrm{supp}(\mathfrak{p}_2)\geq 3$, we must have 
        $\mathfrak{p}_1, \mathfrak{p}_2 \notin \mathrm{Ass}_R(R/J(G))$, and so $\mathfrak{p}_1=\mathfrak{p}_2=\mathfrak{p}$. 
     This leads to a   contradiction. 
   
To establish the converse, let $G$ be either a bipartite graph or an almost bipartite graph. 
According to Corollary \ref{Cover. Bipartite. NTF}, we know that the cover ideal of any bipartite graph is 
normally torsion-free. In fact, in this case, we have $\mathrm{Ass}_R(R/J(G))=\mathrm{Ass}_R(R/J(G))^s$ for all $s\geq 2$, and by choosing 
$\mathfrak{p}:=(x_i~:~ x_i\in V(G))$,  the claim is true. 
Hence, suppose that  $G$ is an almost bipartite graph, and assume that  $C$ is  its unique induced odd cycle subgraph. Set 
  $\mathfrak{p}:=(x_i~:~ x_i\in V(C))$. 
  Following the proof of \cite[Theorem 4.5]{NQKR} implies that $\mathrm{Ass}(J(G)^s)=\mathrm{Ass}_R(R/J(G)) \cup \{\mathfrak{p}\}$ for all $s\geq 2$.  This means that $J(G)$ has nearly copersistence property, as required. 
\end{proof}


In Proposition \ref{NCOP-to-COP} and Lemma \ref{Lem.NCOP.1}, we seek some connections between the copersistence property and nearly copersistence property. 

\begin{proposition} \label{NCOP-to-COP}
Let $I\subset  R=K[x_1, \ldots, x_n]$ be a monomial ideal such that has nearly copersistence property. Then   there exists a monomial prime ideal $\mathfrak{p}\in V^*(I)$ such that $I(\mathfrak{p}\setminus \{x_i\})$ has the copersistence property  for all $x_i\in \mathfrak{p}$.
\end{proposition}

\begin{proof}
Due to $I$ has  nearly copersistence property, hence  there exist  $s\geq 1$ and a monomial prime ideal  $\mathfrak{p}$ such that 
 $\mathrm{Ass}_R(R/I^{m}) \cup \{\mathfrak{p}\} \supseteq \mathrm{Ass}_R(R/I^{m+1})$ for all $1\leq m\leq s$, and  $\mathrm{Ass}_R(R/I^m) \supseteq \mathrm{Ass}_R(R/I^{m+1})$   for all $m \geq s+1$. In the sequel, we want  to show that $I(\mathfrak{p}\setminus \{x_i\})$ has the copersistence property  for all $x_i\in \mathfrak{p}$. Fix  $x_i\in \mathfrak{p}$ and  put $\mathfrak{q}:=\mathfrak{p}\setminus \{x_i\}$. We aim to prove that 
  $\mathrm{Ass}_{R(\mathfrak{q})}(R(\mathfrak{q})/I(\mathfrak{q})^{r}) \supseteq 
  \mathrm{Ass}_{R(\mathfrak{q})}(R(\mathfrak{q})/I(\mathfrak{q})^{r+1})$ for all $r\geq 1$. To do this, fix  $r \geq 1$. From Lemma \ref{LEM.Localization}(ii), 
  we have $I(\mathfrak{q})^{r+1}=I^{r+1}(\mathfrak{q})$. Now, according to Theorem \ref{Localization}, we get the following 
$$\mathrm{Ass}_{R(\mathfrak{q})}(R(\mathfrak{q})/I^{r+1}(\mathfrak{q}))=\{Q~:~ Q\in \mathrm{Ass}_{R}(R/I^{r+1})~\text{and}~ Q \subseteq \mathfrak{q}\}.$$ 
Consider  $Q\in \mathrm{Ass}_{R(\mathfrak{q})}(R(\mathfrak{q})/I(\mathfrak{q})^{r+1})$. Hence, we obtain  $Q\in \mathrm{Ass}_{R}(R/I^{r+1})$ and
  $Q \subseteq \mathfrak{q}$. It follows from  $\mathfrak{q}=\mathfrak{p}\setminus \{x_i\}$ that   $Q\neq \mathfrak{p}$. We therefore must have 
    $Q\in   \mathrm{Ass}_R(R/I^r)$, and Theorem  \ref{Localization} implies that 
     $Q\in  \mathrm{Ass}_{R(\mathfrak{q})}(R(\mathfrak{q})/I^r(\mathfrak{q}))$. Since $I(\mathfrak{q})^{r}=I^{r}(\mathfrak{q})$, this leads to 
      $Q\in  \mathrm{Ass}_{R(\mathfrak{q})}(R(\mathfrak{q})/I(\mathfrak{q})^{r})$. This gives rise to $I(\mathfrak{p}\setminus \{x_i\})$ has the copersistence property, as desired. 
\end{proof}


\begin{lemma} \label{Lem.NCOP.1}
  Let  $I$ be a monomial ideal in $R=K[x_1, \ldots, x_n]$ such that  $I(\mathfrak{m}\setminus \{x_i\})$ has  the copersistence property   for all $i=1, \ldots, n$, where $\mathfrak{m}=(x_1, \ldots, x_n)$. Then $I$    has nearly copersistence property.  
\end{lemma}

 \begin{proof}
   Fix $k \geq 1$ and choose   $Q \in \mathrm{Ass}_R(R/I^{k+1})$ with  $Q\neq \mathfrak{m}$.  Since  $Q$ is a monomial prime ideal, we can derive  that 
   $Q\subseteq \mathfrak{m}\setminus \{x_j\}$ for some $x_j \in \mathfrak{m}$. Set   $\mathfrak{q}:=\mathfrak{m}\setminus \{x_j\}$. 
   Due to  $I(\mathfrak{q})$ has the copersistence property, we get 
     $\mathrm{Ass}_{R(\mathfrak{q})}(R(\mathfrak{q})/I(\mathfrak{q})^k) \supseteq \mathrm{Ass}_{R(\mathfrak{q})}(R(\mathfrak{q})/I(\mathfrak{q})^{k+1})$.  Since   $Q \in \mathrm{Ass}_R(R/I^{k+1})$ and  $Q\subseteq \mathfrak{q}$,      Theorem \ref{Localization} gives that $Q\in \mathrm{Ass}_{R(\mathfrak{q})}(R(\mathfrak{q})/I^{k+1}(\mathfrak{q}))$, and by virtue of 
     $I(\mathfrak{q})^{k+1}=I^{k+1}(\mathfrak{q})$, we obtain $Q\in \mathrm{Ass}_{R(\mathfrak{q})}(R(\mathfrak{q})/I(\mathfrak{q})^{k+1})$.
    As   $I(\mathfrak{q})$ has  the copersistence property, this implies that 
$Q \in \mathrm{Ass}_{R(\mathfrak{q})}(R(\mathfrak{q})/I(\mathfrak{q})^k)$, and hence $Q \in \mathrm{Ass}_{R(\mathfrak{q})}(R(\mathfrak{q})/I^k(\mathfrak{q}))$. It follows  from Theorem \ref{Localization} that $Q \in \mathrm{Ass}_R(R/I^{k})$.
  This leads to $\mathrm{Ass}_R(R/I^{k}) \cup \{\mathfrak{m}\} \supseteq \mathrm{Ass}_R(R/I^{k+1})$. Let $\ell_0$ denote the index of stability of 
  $I$. Then we have   $\mathrm{Ass}_R(R/I^{\alpha}) = \mathrm{Ass}_R(R/I^{\alpha+1})$ for all $\alpha \geq \ell_0$. 
Hence, this permits  us to deduce that   $I$  has nearly copersistence property, and the proof is complete.  
\end{proof}


In the subsequent example,  we will show  how one can use Lemma  \ref{Lem.NCOP.1} to detect whether a monomial ideal has  nearly copersistence property 
or not.  In fact, Morey and  Villarreal  in \cite[Example 4.18]{MV} stated that for 
 $I=(x_1x_2^2x_3, x_2x_3^2x_4, x_3x_4^2x_5, x_4x_5^2x_1, x_5x_1^2x_2)$ and $\mathfrak{m}=(x_1, x_2, x_3, x_4, x_5)$, 
 we have $\mathfrak{m}\in \mathrm{Ass}_R(R/I)$, $\mathfrak{m}\notin \mathrm{Ass}_R(R/I^2)$, 
$\mathfrak{m}\notin \mathrm{Ass}_R(R/I^3)$, and $\mathfrak{m}\in \mathrm{Ass}_R(R/I^4)$. In Example \ref{Exam.NCOP.1}, we show $I$ has 
nearly copersistence property.

\begin{example} \label{Exam.NCOP.1}
\em{
Let   $R=K[x_1,x_2,x_3,x_4,x_5]$, $\mathfrak{m}=(x_1, x_2, x_3, x_4, x_5)$, and consider the monomial ideal
   $I=(x_1x_2^2x_3, x_2x_3^2x_4, x_3x_4^2x_5, x_4x_5^2x_1, x_5x_1^2x_2)$ in $R$. 
We claim that $I$    has nearly copersistence property. To do this, our strategy is to employ Lemma \ref{Lem.NCOP.1}. By the symmetry, it is enough for us to 
show that
 $J:=I(\mathfrak{m}\setminus \{x_5\})=(x_1x_2^2x_3, x_2x_3^2x_4, x_3x_4^2, x_1x_4, x_1^2x_2)$ has  the copersistence property  
 in $S=K[x_1, x_2, x_3, x_4]$.  To reach this aim, we verify that $\mathrm{Ass}_S(S/J^k)=\mathrm{Ass}_S(S/J)$ for all $k\geq 1$. 
 To do this, it is sufficient to prove that   $J$ has the persistence property  and also is normally torsion-free, since both 
 $\mathrm{Ass}_S(S/J^k) \subseteq \mathrm{Ass}_S(S/J^{k+1})$ and $\mathrm{Ass}_S(S/J^k) \subseteq \mathrm{Ass}_S(S/J)$ for all $k\geq 1$ imply that 
  $\mathrm{Ass}_S(S/J^k)=\mathrm{Ass}_S(S/J)$ for all $k\geq 1$.  
 Using  \textit{Macaulay2} \cite{GS}, we are able to check that $J$, $J^2$, and $J^3$ are integrally closed. 
 Since $S$ is a polynomial ring in four  variables,  it follows  from Proposition  \ref{Roberts-Normal}  that $J$ is a normal monomial ideal. In view of the proof of Proposition  \ref{Power}(v), $J$ has the  persistence property. Here, we demonstrate that $J$ is normally torsion-free.  Using \textit{Macaulay2} \cite{GS}, we get 
$$\mathrm{Ass}_S(S/J)=\{(x_1, x_3), (x_1, x_4), (x_2, x_4), (x_1, x_2, x_4), (x_1, x_3, x_4)\}.$$
To complete our argument, we need only establish $(x_1, x_2, x_3) \notin \mathrm{Ass}_S(S/J^k)$, $(x_2, x_3,x_4) \notin \mathrm{Ass}_S(S/J^k)$,  and  $(x_1,x_2, x_3, x_4)\notin \mathrm{Ass}_S(S/J^k)$ 
 for all $k\geq 2$. Suppose, on the contrary, that $\mathfrak{p}_1:=(x_1, x_2, x_3) \in \mathrm{Ass}_S(S/J^k)$ 
 (resp.  $\mathfrak{p}_2:=(x_2, x_3,x_4) \in \mathrm{Ass}_S(S/J^k)$)  for some $k\geq 2$. By  Theorem \ref{Localization}, we get  $(x_1, x_2, x_3) \in \mathrm{Ass}_S(S/J^k(\mathfrak{p}_1))$ (resp. $(x_2, x_3,x_4) \in \mathrm{Ass}_S(S/J^k(\mathfrak{p}_2)))$. 
  Due to $J^k(\mathfrak{p}_1)=J(\mathfrak{p}_1)^{k}$ (resp.  $J^k(\mathfrak{p}_2)=J(\mathfrak{p}_2)^{k}$), we obtain 
   $(x_1, x_2, x_3) \in \mathrm{Ass}_S(S/J(\mathfrak{p}_1)^{k})$ (resp. $(x_2, x_3,x_4) \in \mathrm{Ass}_S(S/J(\mathfrak{p}_2)^{k})$). 
  It is not difficult to see that $J(\mathfrak{p}_1)=(x_1, x_3)$ (resp. $J(\mathfrak{p}_2)=(x_2, x_4)$). Hence, we have 
   $\mathrm{Ass}_S(S/J(\mathfrak{p}_1)^{k})=\{(x_1, x_3)\}$ (resp. $\mathrm{Ass}_S(S/J(\mathfrak{p}_2)^{k})=\{(x_2, x_4)\}$), a contradiction. 
     We therefore have $(x_1, x_2, x_3) \notin \mathrm{Ass}_S(S/J^k)$ and  $(x_2, x_3,x_4) \notin \mathrm{Ass}_S(S/J^k)$ for all $k\geq 2$.  
   On the contrary, assume  $\mathfrak{p}_3:=(x_1, x_2, x_3,x_4) \in \mathrm{Ass}_S(S/J^k)$  for some $k\geq 2$, and seek a contradiction. 
   Since $J$ has the persistence property, we must have $\mathfrak{p}_3 \in \mathrm{Ass}^{\infty}(J)$, where 
     $\mathrm{Ass}^{\infty}(J)$ denotes the  stable set  of associated prime ideals  of $J$. We can now borrow the 
algorithm in \cite[page 216]{BHR} to compute $\mathrm{Ass}^{\infty}(J)$ as follows. In particular, it should be noted that since
 $\mathrm{ht} (\mathfrak{p}_3)=4$, 
we have to consider  the 3rd Koszul homology. 
\begin{verbatim} 
S = QQ[x_1, x_2, x_3, x_4]; 
J = monomialIdeal(x_1*x_2^2*x_3, x_2*x_3^2*x_4, x_3*x_4^2,
x_1*x_4, x_1^2*x_2); 
R = reesAlgebra(J); 
P_3 = {x_1,x_2,x_3,x_4};
phi = matrix{P_3}; 
C = (koszul phi)**R; 
dim(HH_3 C) > 0
 \end{verbatim}
After using  \textit{Macaulay2} \cite{GS}, we see that  the output of the above algorithm is: $\mathtt{false}$. This means that $\mathfrak{p}_3 \notin \mathrm{Ass}^{\infty}(J)$, which leads to a contradiction.  Consequently, we have $(x_1, x_2, x_3,x_4) \notin \mathrm{Ass}_S(S/J^k)$  for all $k\geq 2$, and 
so $J$ is normally torsion-free. This gives rise to $J$ has the  copersistence property, and therefore $I$ has nearly copersistence property, as desired. 
}
\end{example}


 Already, Herzog and Vladoiu in \cite{HV} introduced a class of monomial ideals which are called  monomial ideals  of intersection type. 
 Indeed, a monomial ideal is said to be a {\it monomial ideal of intersection type} when it can be presented as an intersection of powers of monomial prime ideals. In Example \ref{Exam.NCOP.2}, by using Lemma  \ref{Lem.NCOP.1},  we will investigate a subclass of monomial ideals of intersection type which have nearly copersistence property.

\begin{example} \label{Exam.NCOP.2}
\em{
 Let  $I\subset R=K[x_1, \ldots, x_n]$  be the following  monomial ideal   
  $$I=(\mathfrak{m}\setminus \{x_1\})^{d_1} \cap (\mathfrak{m}\setminus \{x_2\})^{d_2} \cap \cdots \cap  (\mathfrak{m}\setminus \{x_n\})^{d_n},$$ 
  where $\mathfrak{m}=(x_1, \ldots, x_n)$ and  $d_1, \ldots, d_n$ are some positive integers. 
 It is easy to see that  
 $I(\mathfrak{m}\setminus \{x_i\})=(\mathfrak{m}\setminus \{x_i\})^{d_i}$  for all $i=1, \ldots, n$.   
 By virtue of  $\mathfrak{m}\setminus \{x_i\}$, for each $i=1, \ldots, n$, has the copersistence property, we can derive from 
 Proposition  \ref{Power}(ii) that   $(\mathfrak{m}\setminus \{x_i\})^{d_i}$, for each $i=1, \ldots, n$, has the copersistence property as well.
 Now, it  follows immediately from  Lemma \ref{Lem.NCOP.1} that $I$ has nearly copersistence property. 
}
\end{example}


It has already been shown in \cite[Lemma 3.1]{CMS} that if $G$  is a cycle of length $2k + 1$ and $I$ is the edge ideal of $G$, 
then  $\mathrm{Ass}(R/I^n) = \mathrm{Min}(R/I)$ if $n \leq  k$ and  $\mathrm{Ass}(R/I^n) =\mathrm{Min}(R/I) \cup \{\mathfrak{m}\}$ 
if  $n \geq  k + 1.$ It follows promptly  from this lemma  that $I$ has nearly copersistence property. In the next corollary, we will directly re-prove this  fact
 by means of  Lemma  \ref{Lem.NCOP.1}. 
  
\begin{corollary} \label{NCOP.Edge}
Edge ideals of odd cycles have nearly copersistence property.
\end{corollary}

\begin{proof}
Assume that  $R=K[x_1, \ldots, x_{2n+1}]$,  $\mathfrak{m}=(x_1, \ldots, x_{2n+1})$, and  $C_{2n+1}=(V(C_{2n+1}), E(C_{2n+1}))$ denotes 
 the odd cycle graph of order $2n+1$ with the vertex set $V(C_{2n+1})=\{x_1, \ldots, x_{2n+1}\}$ and  the following edge set
$$E(C_{2n+1})=\{\{x_i, x_{i+1}\} \mid i=1, \ldots, 2n\} \cup \{\{x_{2n+1}, x_{1}\}\}.$$ Put $L:=I(C_{2n+1})$. 
By the symmetry, it is sufficient to demonstrate that $L(\mathfrak{m}\setminus \{x_{2n+1}\})$ has the copersistence property. One can easily see that 
$$L(\mathfrak{m}\setminus \{x_{2n+1}\})=(x_ix_{i+1} \mid i=2, \ldots, 2n-2)+ (x_1, x_{2n}).$$
Set $J:=(x_ix_{i+1} \mid i=2, \ldots, 2n-2)$. Since $J$  is the edge ideal of a path graph, we can deduce that 
$\mathrm{Ass}_R(R/J)=\mathrm{Ass}_R(R/J^s)$ for all $s\geq 1$, and so has the copersistence property. Moreover, the monomial prime ideal 
$\mathfrak{q}:=(x_1, x_{2n})$ has the copersistence property. Due to $\mathrm{supp}(\mathfrak{q}) \cap \mathrm{supp}(J)=\emptyset$, it follows from 
Theorem \ref{Summation} that $L$ has the copersistence property. Now, the claim can be concluded from Lemma \ref{Lem.NCOP.1}. 
\end{proof}


It is natural to ask whether the converse of  Lemma  \ref{Lem.NCOP.1} is true. The answer, in general, is negative  and we provide such a counterexample in 
  Example \ref{Not-Lem.NCOP.1}.  Before stating it, one has to  remember the following definition and fact. 
 
\begin{definition} (\cite[Definition 10.5.2]{V1}) 
\em{
A \textit{Hochster configuration of order $k$} of $G$ consists of two 
odd cycles $C_{2r+1}$ and $C_{2s+1}$ satisfying the following conditions:
\begin{itemize}
\item[(i)]  $C_{2r+1} \cap  N_G(C_{2s+1}) = \emptyset$  and $k = r + s + 1$.
\item[(ii)] No chord of either $C_{2r+1}$ or $C_{2s+1}$ is an edge of $G$.
\end{itemize}
}
\end{definition}

\begin{proposition} (\cite[Proposition 10.5.3]{V1}) \label{Hochster}
 If $G$ has a Hochster configuration of order $k$, then $\overline{I(G)^k}\neq I(G)^k$.
\end{proposition}

\begin{example} \label{Not-Lem.NCOP.1} 
\em{
Let $R=K[x_1, x_2, x_3, x_4, x_5, x_6]$ be a polynomial ring over a field $K$ and  $L$ be the following square-free  monomial ideal
$$L=(x_1x_2, x_2x_3, x_3x_4, x_4x_5, x_5x_1, x_1x_4x_6).$$ 
Set   $\mathfrak{m}:=(x_1, x_2, x_3, x_4, x_5, x_6)$  and $\mathfrak{p}:=(x_1, x_2, x_3, x_4, x_5)$. 
We claim that $L$ has nearly copersistence property,  while 
$$Q:=L(\mathfrak{m}\setminus\{x_6\})=(x_1x_2, x_2x_3, x_3x_4, x_4x_5, x_5x_1, x_1x_4),$$
 does not have the copersistence property. Using  \textit{Macaulay2} \cite{GS}, we find
 \begin{align*}
 \mathrm{Ass}_R(R/L)=\{&(x_1, x_2, x_4), (x_1,x_3,x_4), (x_1,x_3,x_5),  (x_2,x_4,x_5), \\
 & (x_2,x_3,x_5,x_6)\},
 \end{align*}
  and $\mathrm{Ass}_R(R/L^2)=\mathrm{Ass}(R/L)\cup \{\mathfrak{p}\}$.   In what follows, our aim is to show that
  $\mathrm{Ass}_R(R/L^s)= \mathrm{Ass}(R/L) \cup \{\mathfrak{p}\}$ for all $s\geq 3$. To simplify the notation, put 
  $I:=(x_1x_2, x_2x_3, x_3x_4, x_4x_5, x_5x_1).$ 
  According to  Proposition \ref{Hochster}, we conclude  that both $I$ and $J:=I+(x_1x_4)$ are  normal. It follows now from Theorem \ref{ANKRQ} 
  that $L=I+ (x_1x_4x_6)$ is normal as well. According to the proof of Proposition \ref{Power}(v), we can deduce that $L$ has the persistence property. 
 This implies that $\mathfrak{p}\in \mathrm{Ass}_R(R/L^s)$ for all $s\geq 3$.
  We here verify that $\mathfrak{m}\notin \mathrm{Ass}_R(R/L^s)$ for all $s\geq 3$. 
  Otherwise, $\mathfrak{m}\in \mathrm{Ass}_R(R/L^s)$ for some  $s\geq 3$. 
Since $L$ has the persistence property, this implies that we must have $\mathfrak{m}\in \mathrm{Ass}^{\infty}(L)$. 
 We can now refer to the algorithm in \cite[page 216]{BHR} to determine   $\mathrm{Ass}^{\infty}(L)$ as follows. Furthermore, notice that  since 
 $\mathrm{ht} (\mathfrak{m})=6$, we must  consider  the $5$th Koszul homology. 
\begin{verbatim} 
S = QQ[x_1, x_2, x_3, x_4, x_5, x_6]; 
L = monomialIdeal(x_1*x_2, x_2*x_3, x_3*x_4, x_4*x_5,
x_5*x_1, x_1*x_4*x_6); 
R = reesAlgebra(L); 
m = {x_1,x_2,x_3,x_4,x_5,x_6};
phi = matrix{m}; 
C = (koszul phi)**R; 
dim(HH_5 C) > 0
 \end{verbatim}
Using  \textit{Macaulay2} \cite{GS}, we detect that  the output of the above algorithm is: $\mathtt{false}$. 
This gives  that $\mathfrak{m} \notin \mathrm{Ass}^{\infty}(L)$,   a contradiction. Therefore, we have 
 $\mathfrak{m}\notin \mathrm{Ass}(R/L^s)$ for all $s\geq 3$.  From now on, our strategy is to rely on Theorem  \ref{Localization}, which
  enables us to eliminate redundant choices.   Let $\mathfrak{q}\in \mathrm{Ass}_R(R/L^s)\setminus \mathrm{Ass}_R(R/L^2)$ for some $s\geq 3$ be an embedded primes. Based on the above discussion, we have  $\mathfrak{q}\neq \mathfrak{m}$ and $\mathfrak{q}\neq \mathfrak{p}$. Hence, 
    we gather  all  possible cases in the table below: 
 
 \begin{center}
 \begin{tabular}{|c|c|}
  \hline
   $\mathfrak{q}=(x_2,x_3,x_4,x_5)$  &    $L(\mathfrak{q})=(x_2,x_4,x_5)$  \\
\hline
  $\mathfrak{q}=(x_2,x_4,x_5,x_6)$   &     $L(\mathfrak{q})=(x_2,x_4,x_5)$   \\
\hline
 $\mathfrak{q}=(x_1,x_2,x_4,x_6)$   &     $L(\mathfrak{q})=(x_1, x_2, x_4)$   \\
\hline
  $\mathfrak{q}=(x_1,x_3,x_4,x_6)$  &    $L(\mathfrak{q})=(x_1, x_3, x_4)$  \\
\hline
   $\mathfrak{q}=(x_1,x_3,x_5,x_6)$ &  $L(\mathfrak{q})=(x_1, x_3, x_5)$  \\
  \hline  
    $\mathfrak{q}=(x_1,x_2,x_3,x_4)$  & $L(\mathfrak{q})=(x_1, x_4, x_2x_3)$ \\
  \hline
   $\mathfrak{q}=(x_1,x_2,x_3,x_5)$ &    $L(\mathfrak{q})=x_1(x_2,x_4)+(x_3,x_5)$   \\
 \hline
 $\mathfrak{q}=(x_1,x_2,x_4,x_5)$  &  $L(\mathfrak{q})=(x_2, x_4, x_1x_5)$   \\
\hline
 $\mathfrak{q}=(x_1,x_3,x_4,x_5)$ &     $L(\mathfrak{q})=(x_1, x_3, x_4x_5)$ \\
\hline
     $\mathfrak{q}=(x_1,x_2,x_3,x_4,x_6)$ & $L(\mathfrak{q})=(x_1,x_4, x_2 x_3)$   \\
\hline
   $\mathfrak{q}=(x_1,x_2,x_4,x_5,x_6)$ & $L(\mathfrak{q})=(x_2,x_4, x_1 x_5)$  \\
\hline
  $\mathfrak{q}=(x_1,x_3,x_4,x_5,x_6)$ & $L(\mathfrak{q})=(x_1,x_3, x_4 x_5)$  \\
\hline
    $\mathfrak{q}=(x_1,x_2,x_3,x_5,x_6)$ & $L(\mathfrak{q})=x_1(x_2,x_6)+(x_3,x_5)$ \\
\hline
  $\mathfrak{q}=(x_2,x_3,x_4,x_5,x_6)$ &  $L(\mathfrak{q})=x_4(x_3,x_6)+(x_2,x_5)$  \\
\hline
   \end{tabular}
\end{center}

\bigskip
On account of Theorem \ref{Th.NTF.Summation},  each of the $L(\mathfrak{q})$'s  is a normally torsion-free square-free monomial ideal. 
This leads to  a contradiction. Thus,  we obtain    $\mathrm{Ass}_R(R/L^s)= \mathrm{Ass}(R/L) \cup \{\mathfrak{p}\}$ for all $s\geq 3$.
In  light of  $\mathrm{Ass}_R(R/L)\cup \{\mathfrak{p}\} \supseteq \mathrm{Ass}_R(R/L^2)$, and for all $m\geq 2$, we have 
 $\mathrm{Ass}_R(R/L^m) \supseteq \mathrm{Ass}_R(R/L^{m+1})$, this implies that $L$ has nearly copersistence property. 
   On the other hand, using  \textit{Macaulay2} \cite{GS}, we get    $(x_1, x_2, x_3, x_4, x_5)\in \mathrm{Ass}(Q^2)\setminus \mathrm{Ass}(Q).$
   Hence, $Q$ does not have the copersistence property. This yields that  the converse of  Lemma  \ref{Lem.NCOP.1} is not true in general. 
 }
\end{example}


The following  proposition asserts that, under a certain condition,  a monomial ideal has nearly copersistence property if and only if its monomial multiple  has 
nearly  copersistence property.

\begin{proposition}\label{Pro.NCOP.Multiple}
Let  $I \subset R=K[x_1, \ldots, x_n]$ be a monomial ideal, and $h$  be a monomial in $R$ such that $\mathrm{gcd}(h,u)=1$ for all $u\in \mathcal{G}(I)$. Then  $I$  has nearly copersistence property  if and only if $hI$   has nearly copersistence property. 
\end{proposition}

\begin{proof}
To show the forward implication, suppose that  $I$ has  nearly copersistence property. Also,   let $h=x_{j_1}^{\alpha_1}\cdots x_{j_t}^{\alpha_t}$ 
with $j_1, \ldots, j_t \in \{1, \ldots, n\}$ and $\alpha_i>0$ for all $i=1, \ldots, t$.   It follows from Theorem   \ref{5.2KHN} that   
\begin{equation}
\mathrm{Ass}_R(R/(hI)^{m+1})=\mathrm{Ass}_R(R/I^{m+1})\cup\{ (x_{j_1}),\ldots,(x_{j_t})\}, \label{NCOP.Multiple.1}
\end{equation}

and
\begin{equation}
\mathrm{Ass}_R(R/(hI)^{m})=\mathrm{Ass}_R(R/I^{m})\cup\{ (x_{j_1}),\ldots,(x_{j_t})\}. \label{NCOP.Multiple.2}
\end{equation}
Since  $I$ has  nearly copersistence property,  hence  there exist  $s\geq 1$ and a monomial prime ideal  $\mathfrak{p}$ such that  $\mathrm{Ass}_R(R/I^{m}) \cup \{\mathfrak{p}\} \supseteq \mathrm{Ass}_R(R/I^{m+1})$ for all $1\leq m\leq s$, and  $\mathrm{Ass}_R(R/I^m) \supseteq \mathrm{Ass}_R(R/I^{m+1})$   for all $m \geq s+1$. 
  If $1\leq m \leq s$, then $\mathrm{Ass}_R(R/I^{m}) \cup \{\mathfrak{p}\} \supseteq \mathrm{Ass}_R(R/I^{m+1})$, and so 
   $$\mathrm{Ass}_R(R/I^{m})\cup\{ (x_{j_1}),\ldots,(x_{j_t})\} \cup \{\mathfrak{p}\} \supseteq \mathrm{Ass}_R(R/I^{m+1})\cup\{ (x_{j_1}),\ldots,(x_{j_t})\}.$$
  Based on     \eqref{NCOP.Multiple.1} and \eqref{NCOP.Multiple.2}, we obtain $\mathrm{Ass}_R(R/(hI)^{m})  \cup \{\mathfrak{p}\} \supseteq \mathrm{Ass}_R(R/(hI)^{m+1})$. 
 Hence, let  $m\geq s+1$.  Since $\mathrm{Ass}_R(R/I^m) \supseteq \mathrm{Ass}_R(R/I^{m+1})$, we get 
 $$\mathrm{Ass}_R(R/I^{m})\cup\{ (x_{j_1}),\ldots,(x_{j_t})\}  \supseteq \mathrm{Ass}_R(R/I^{m+1})\cup\{ (x_{j_1}),\ldots,(x_{j_t})\}.$$
 It  follows rapidly from \eqref{NCOP.Multiple.1} and  \eqref{NCOP.Multiple.2} that  $\mathrm{Ass}_R(R/(hI)^{m}) \supseteq \mathrm{Ass}_R(R/(hI)^{m+1})$. 
 This implies that $hI$   has nearly copersistence property.

To demonstrate the converse, assume  that $hI$   has nearly copersistence property.  This yields that  there exist $s\geq 1$  and a monomial prime ideal  $\mathfrak{p}$ such that   $\mathrm{Ass}_R(R/(hI)^{m}) \cup \{\mathfrak{p}\} \supseteq \mathrm{Ass}_R(R/(hI)^{m+1})$ for all $1\leq m\leq s$, and 
 $\mathrm{Ass}_R(R/(hI)^m) \supseteq \mathrm{Ass}_R(R/(hI)^{m+1})$   for all $m \geq s+1$. 
 Take  an arbitrary element $\mathfrak{q}\in \mathrm{Ass}_R(R/I^{m+1})$. 
 Let  $1\leq m \leq s$. From \eqref{NCOP.Multiple.1}, we have   
 \begin{equation}
 \mathrm{Ass}_R(R/I^{m})\cup\{ (x_{j_1}),\ldots,(x_{j_t})\}\cup \{\mathfrak{p}\}
   \supseteq \mathrm{Ass}_R(R/I^{m+1})\cup\{ (x_{j_1}),\ldots,(x_{j_t})\}. \label{NCOP.Multiple.3}
 \end{equation}
  According to \eqref{NCOP.Multiple.3}, we derive that   $\mathfrak{q}\in  \mathrm{Ass}_R(R/I^{m})\cup\{ (x_{j_1}),\ldots,(x_{j_t})\}\cup \{\mathfrak{p}\}$. 
 By virtue of  $\mathrm{gcd}(h,u)=1$ for all $u\in \mathcal{G}(I)$,   Proposition \ref{Pro.supp} implies that  $\mathfrak{q} \notin  \{ (x_{j_1}),\ldots,(x_{j_t})\}$. This gives rise to    $\mathfrak{q}\in  \mathrm{Ass}_R(R/I^{m})\cup \{\mathfrak{p}\}$.  We thus   assume that $m\geq s+1$. 
On account of   $\mathfrak{q} \notin  \{ (x_{j_1}),\ldots,(x_{j_t})\}$ and    $$\mathrm{Ass}_R(R/I^{m})\cup\{ (x_{j_1}),\ldots,(x_{j_t})\}    \supseteq \mathrm{Ass}_R(R/I^{m+1})\cup\{ (x_{j_1}),\ldots,(x_{j_t})\},$$ one can promptly conclude that  $\mathfrak{q}\in  \mathrm{Ass}_R(R/I^{m})$. 
   This implies that $I$  has nearly copersistence property, and the proof is complete. 
\end{proof}


In the subsequent  lemma, we  prove  that  a monomial ideal has nearly  copersistence property if and only if its expansion (or weighted)  has nearly 
 copersistence property. 

\begin{lemma}\label{Lem.NCOP.Expansion+Weighting} 
Let  $I \subset R=K[x_1, \ldots, x_n]$ be a monomial ideal. Then the following statements hold.
\begin{itemize}
\item[(i)] $I$ has nearly copersistence property if and only if $I^*$ has  nearly copersistence property, 
where $I^*$ denotes the  expansion of $I$.  
\item[(ii)] $I$ has nearly copersistence property if and only if $I_W$ has nearly  copersistence property, 
where $I_W$ denotes the  weighted ideal.    
\end{itemize}
\end{lemma}

\begin{proof}
(i) We first assume that $I$ has nearly  copersistence property. Choose $\mathfrak{q}\in \mathrm{Ass}_{R^*}(R^*/(I^*)^{k+1})$, where $k\geq 1$. 
It follows from Lemma \ref{Lem.Bayati.Expansion}(iii) that   $(I^*)^{k+1}=(I^{k+1})^*$, and hence  $\mathfrak{q}\in \mathrm{Ass}_{R^*}(R^*/(I^{k+1})^*)$. 
From   Proposition \ref{Pro.Bayati.Expansion}, we get    $\mathfrak{p}_1\in \mathrm{Ass}_R(R/I^{k+1})$, where $\mathfrak{q}=\mathfrak{p}_1^*$. 
 Thanks to  $I$ has nearly  copersistence property, this gives that  there exist a positive integer $s$ and a monomial prime ideal  $\mathfrak{p}$ such that 
 $\mathrm{Ass}_R(R/I^{m}) \cup \{\mathfrak{p}\} \supseteq \mathrm{Ass}_R(R/I^{m+1})$ for all $1\leq m\leq s$, and  $\mathrm{Ass}_R(R/I^m) \supseteq \mathrm{Ass}_R(R/I^{m+1})$   for all $m \geq s+1$. Let $k \geq s+1$. Then  $\mathfrak{p}_1\in \mathrm{Ass}_R(R/I^{k})$. Now,  Proposition \ref{Pro.Bayati.Expansion} implies that    $\mathfrak{q}\in \mathrm{Ass}_{R^*}(R^*/(I^k)^*)$. Since  $(I^*)^{k}=(I^{k})^*$, this yields that  $\mathfrak{q}\in \mathrm{Ass}_{R^*}(R^*/(I^*)^k)$.   So, let $1\leq k\leq s$. If $\mathfrak{p}_1\in \mathrm{Ass}_R(R/I^{k})$, then  we can repeat the previous argument.  Hence, let $\mathfrak{p}_1=\mathfrak{p}$. Then 
  $\mathfrak{q}\in \mathrm{Ass}_{R^*}(R^*/(I^*)^k) \cup  \{\mathfrak{p}^*\}$. Therefore,  $I^*\subset R^*$ has nearly copersistence property.   
 By a similar argument,  we can establish the converse.

(ii) This statement can be shown by combining  together Lemmas  \ref{LEM. Weighted} and \ref{ASS-Weighted} and mimicking the proof of part (i). 
\end{proof}


As an application of  Lemma  \ref{Lem.NCOP.Expansion+Weighting}, we can provide the following proposition, which ensures that there exist infinitely many monomial ideals possessing  nearly copersistence property.

\begin{proposition} \label{Pro.Infinite.NCOP}
Let $R=K[x_1, \ldots, x_n]$ be  a polynomial ring in $n$  variables  with coefficients in a  field $K$. Then there exist infinitely many 
 monomial ideals  in $R$ such that  have  nearly  copersistence property.
\end{proposition}

\begin{proof}
By virtue of  $(x_1) \subset R=K[x_1]$ and $(x_1x_2) \subset R=K[x_1,x_2]$ are normally torsion-free square-free monomial ideals, 
 they have nearly copersistence property. Hence, we can  assume   $n\geq 3$. Let  $L := (x_1x_2, x_2x_3, x_1x_3)$ in $S=K[x_1, x_2, x_3]$. Because  $L=(x_1, x_2) \cap (x_2, x_3) \cap (x_1,x_3)$, 
we can interpret $L$ as the cover ideal of the odd cycle $C_3$, which is an almost bipartite graph. It follows now from Theorem
    \ref{NCOP.Almost} that $L$ has nearly copersistence property.
  Now, let $\mathfrak{p}_1:=(x_{i_1}, \ldots, x_{i_a})$, $\mathfrak{p}_2:=(x_{i_{a+1}}, \ldots, x_{i_b})$, and $\mathfrak{p}_3:=(x_{i_{b+1}}, \ldots, x_{i_c})$ with  $\mathrm{supp}(\mathfrak{p}_i) \cap  \mathrm{supp}(\mathfrak{p}_j)=\emptyset$ for all $1\leq i<j \leq 3$ and $\bigcup_{i=1}^{3}\mathrm{supp}(\mathfrak{p}_i)=\{x_1, \ldots, x_n\}$. 
 We can rapidly conclude  from Theorem \ref{Lem.NCOP.Expansion+Weighting}(i) that  $L^*$ in $R=K[x_1, \ldots, x_n]$
 has nearly copersistence property. Suppose now that   $I:=(L^*)_W$ with $W(x_i)=\alpha_i$ such that $\alpha_i \geq 1$ for all $i=1, \ldots, n$. 
 In view of Theorem \ref{Lem.NCOP.Expansion+Weighting}(ii), we can derive that  $I$  has nearly copersistence property. Consequently,  there exist
  infinitely many such monomial ideals, as desired. 
  \end{proof}


We close this section with the following question: Let $I$ be a monomial ideal in $R=K[x_1, \ldots, x_n]$ such that 
$I=I_1R + I_2R$, where
$\mathcal{G}(I_1) \subset R_1=K[x_1, \ldots, x_m]$ and $\mathcal{G}(I_2) \subset R_2=K[x_{m+1}, \ldots, x_n]$ for some $m \geq 1$. 
 If $I_1$ and $I_2$ have nearly copersistence property,   can we deduce that $I$ has nearly  copersistence property? 
 We here present  a counterexample. In particular, this demonstrates that Proposition \ref{Summation} is not generally  valid for  nearly copersistence property.
 \begin{example}
 \em{
 Let $I=(x_1x_2, x_2x_3, x_3x_4, x_4x_5, x_5x_6, x_6x_7, x_7x_1, x_9x_8)$ be a monomial ideal in $R=K[x_1,x_2,x_3, x_4, x_5, x_6, x_7,x_8,x_9]$. 
 Then we can write $I=I_1+I_2$, where $I_1=(x_1x_2, x_2x_3, x_3x_4, x_4x_5, x_5x_6, x_6x_7, x_7x_1)$ and $I_2=(x_9x_8)$. 
 Since $I_1$ is the edge ideal of a cycle graph of order $7$, we can derive from Corollary \ref{NCOP.Edge} that $I_1$ has nearly copersistence property. 
 Besides, it follows from $\mathrm{Ass}(I_2^s)=\mathrm{Ass}(I_2)=\mathrm{Min}(I_2)$, for all $s\geq 1$, that $I_2$ has nearly copersistence property. 
 On the other hand, using  \textit{Macaulay2} \cite{GS}, we detect that $\mathrm{Ass}_R(R/I)=\mathrm{Ass}_R(R/I^2)=\mathrm{Ass}_R(R/I^3)$, but 
\begin{align*}
\mathrm{Ass}_R(R/I^4)= & \mathrm{Ass}_R(R/I^3) \cup \{(x_1,x_2,x_3,x_4,x_5,x_6,x_7,x_8),\\
& (x_1,x_2,x_3,x_4,x_5,x_6,x_7,x_9)\}.
\end{align*} 
 This shows   that $I$ does not have nearly copersistence property. 
 }
 \end{example}

\section{Conclusion and Outlook} 
In this paper we considered two classes of monomial ideals which have the copersistence property or nearly copersistence property, and  explored their properties.  
This study, in practice, serves as a starting point, raising several questions that could be further investigated. One of the most intriguing concepts might be a generalized version of  nearly copersistence property. Specifically, we define a monomial ideal $I$  to have  \textit{generalized nearly copersistence property}  if there exist a positive integer $s$ and monomial prime ideals  $\mathfrak{p}_1, \ldots, \mathfrak{p}_t$ such that 
 $\mathrm{Ass}_R(R/I^{m}) \cup \{\mathfrak{p}_1, \ldots, \mathfrak{p}_t\} \supseteq \mathrm{Ass}_R(R/I^{m+1})$ for all $1\leq m\leq s$, and  $\mathrm{Ass}_R(R/I^m) \supseteq \mathrm{Ass}_R(R/I^{m+1})$   for all $m \geq s+1$. To see a concrete  example, we consider the following monomial ideal 
  $$I=(x_1^6, x_1^5x_2, x_1x_2^5, x_2^6, x_1^4x_2^4x_3, x_1^4x_2^4x_4, x_1^4x_5^2x_6^3, x_2^4x_5^3x_6^2),$$ 
 in the polynomial ring $R=K[x_1, x_2, x_3, x_4, x_5, x_6]$ from \cite[page 549]{HH2}. Put  $\mathfrak{p}_1=(x_1, x_2, x_5, x_6)$ and  
 $\mathfrak{p}_2=(x_1, x_2, x_3, x_4, x_5,x_6)$.   Using  \textit{Macaulay2} \cite{GS}, we get the following 
  \begin{align*}
\mathrm{Ass}_R(R/I)=\{&(x_1,x_2), (x_1, x_2, x_5), (x_1, x_2, x_6), (x_1, x_2, x_3, x_4, x_5),\\
& (x_1, x_2, x_3, x_4, x_6), \mathfrak{p}_2\},
\end{align*}
 \begin{align*}
\mathrm{Ass}_R(R/I^2)=\{&(x_1,x_2), (x_1,x_2,x_5), (x_1, x_2, x_6), \mathfrak{p}_1,  (x_1, x_2, x_3, x_4, x_5), \\
& (x_1, x_2, x_3, x_4, x_6)\},  
\end{align*}
$$\mathrm{Ass}_R(R/I^3)=\{(x_1,x_2), (x_1,x_2,x_5), (x_1, x_2, x_6), \mathfrak{p}_1, \mathfrak{p}_2\},$$
 and 
 $$\mathrm{Ass}_R(R/I^4)=\{(x_1,x_2), (x_1,x_2,x_5), (x_1, x_2, x_6), \mathfrak{p}_1\}.$$
 Hence, we have 
 $\mathrm{Ass}_R(R/I^{m}) \cup \{\mathfrak{p}_1, \mathfrak{p}_2\} \supseteq \mathrm{Ass}_R(R/I^{m+1})$ for all $1\leq m\leq 2$, and  $\mathrm{Ass}_R(R/I^m) \supseteq \mathrm{Ass}_R(R/I^{m+1})$   for all $3\leq m \leq 10$. 
 Herzog and Hibi \cite{HH2}, in 2005, investigated the above example  from the depth function perspective. After that, in 2014, Bandari, Herzog, and Hibi presented an extended version of this example, see \cite[Theorem 1]{BHH}. Next, in 2024, Rissner and  Swanson generalized this result, refer to 
 \cite{RS} for more details. 
 
 Unfortunately, there is no information available in the literature about monomial ideals that satisfy generalized nearly copersistence property. Therefore, the authors hope that this paper will highlight and promote further research in this area.



\bigskip
\centerline{\bf  The conflict of interest and data availability statement}

\vspace{4mm}
We hereby declare that this manuscript has no associated data and that there is no conflict of interest regarding its content. 
 

\bigskip
\begin{center}
{\bf  ORCID}
  \vspace{3mm}
  \hspace{2cm}  \item[\textbf{ Mehrdad Nasernejad}]: ORCID: https://orcid.org/0000-0003-1073-1934
\hspace{2cm}   \item[\textbf{Jonathan Toledo}]: ORCID: https://orcid.org/0000-0003-3274-1367
 \end{center}




\end{document}